\documentclass[12pt,a4paper]{amsart}
\usepackage{xypic}

\title{The BMR freeness conjecture for the 2-reflection groups}

\author{Ivan Marin}
\author{G\"otz Pfeiffer}
\date{January 9, 2016}

\address{I.M.: LAMFA, Universit\'e de Picardie Jules Verne, 33 rue Saint Leu, 80039 Amiens, France}
\email{ivan.marin@u-picardie.fr}
\address{G.P: School of Mathematics, Statistics and Applied Mathematics, NUI Galway, University Road, Galway, Ireland}
\email{goetz.pfeiffer@nuigalway.ie}

\usepackage[euler-digits]{eulervm}

\usepackage{amssymb} \let\emptyset\varnothing

\usepackage[colorlinks=true,
linkcolor=green,
filecolor=brown,
citecolor=green]{hyperref}

\usepackage[margin=1in,
centering]{geometry}

\usepackage{etoolbox}
\makeatletter
\let\ams@starttoc\@starttoc
\usepackage{parskip}
\let\@starttoc\ams@starttoc
\patchcmd{\@starttoc}{\makeatletter}{\makeatletter\parskip\z@}{}{}
\makeatother

\usepackage{tikz}
\usetikzlibrary{backgrounds,fit}

\newtheorem{theorem}{Theorem}[section]
\newtheorem{Lemma}[theorem]{Lemma}
\newtheorem{cor}[theorem]{Corollary}
\newtheorem{prop}[theorem]{Proposition}
\newtheorem*{conj}{Conjecture}

\newcommand{\Z}{\mathbf{Z}}

\newcommand{\Fail}{\mbox{\!{\Large $\times$}}} 

\newcommand{\revert}{\mathrm{revert}}

\newcommand{\Size}[1]{\left| #1 \right|}
\newcommand{\Span}[1]{\left< #1 \right>}

\tikzset{every picture/.style = {inner sep=0pt,baseline}}
\tikzset{p/.style = {draw, shape          = circle,
                                 text           = black,
                                 font=\tiny,
                                 inner sep      = 0pt,
                                 outer sep      = 0pt,
                                 minimum size   = 1pt}}

\tikzset{r/.style = {red, very thin}}
\tikzset{g/.style = {green, very thin}}
\tikzset{b/.style = {blue, very thin}}
\tikzset{o/.style = {orange, very thin}}
\tikzset{u/.style = {black, very thin}}

\numberwithin{equation}{section}

\begin{document}

\begin{abstract} 
  We prove the freeness conjecture of Brou\'e, Malle and Rouquier for
  the Hecke algebras associated to the primitive complex
  2-reflection groups with a single conjugacy class of reflections.
\end{abstract}

\maketitle

\tableofcontents

\section{Introduction}

We prove several new cases of the freeness conjecture for the generic Hecke algebras associated to
complex reflection groups (sometimes called : cyclotomic Hecke algebras), including all 2-reflection groups (of exceptional types). Recall that, when $W$
is a finite reflection group over the real numbers, that is to say a finite Coxeter group, the Iwahori-Hecke algebra $H$ associated to it can be defined
as a quotient of the group algebra $\Z[q,q^{-1}]B$ of the braid group $B$ associated to $W$ -- which is also known in this setup as an Artin group, or
Artin-Tits group, or Artin-Brieskorn group. This is the quotient by the relations $(s+1)(s-q) = 0$, where $s$ runs among the
natural generators of $B$ -- or equivalently all their conjugates in $B$. These conjugates are called braided reflections.

In the more general setting of complex reflection groups, there is a natural geometric description of these braided reflections, as well as a topological
description of the braid group $B$, described in \cite{BMR}. In case $W$ is generated by (pseudo-)reflections of order more than $2$,
or if $W$ admits several reflection classes (aka conjugacy classes of reflections)
the ring $\Z[q,q^{-1}]$ needs to be replaced by a larger ring. However, since the groups we are interested in are
generated by reflections of order $2$ -- although they can not be realized inside a real form of the vector space -- and have a single reflection class we can and will restrict to this case.
A conjecture of Brou\'e, Malle and Rouquier
in \cite{BMR} then states the following. 

\begin{conj}
The Hecke algebra $H$ defined as the quotient of $\Z[q,q^{-1}] B$ by the relations  $(s+1)(s-q) = 0$ where $s$
runs among the braided reflections of $B$ is a free $\Z[q,q^{-1}]$-module of rank the order $|W|$ of $W$.
\end{conj}
We refer the reader to \cite{CYCLO} for the state-of-the-art of this conjecture, as well as the proof that this formulation of the conjecture is equivalent to
a few others (see proposition 2.9 there). We only mention the following important fact, originally proved in \cite{BMR}.

\begin{prop} \label{propHeckeBMR} In order to prove the conjecture for $W$ it is sufficient to show that $H$ is spanned by $|W|$ elements.
\end{prop}

We state our main result.
\begin{theorem} \label{th:maintheo} All primitive irreducible complex 2-reflection groups with a single reflection class 
satisfy the freeness conjecture, namely
$H$ is a free $\Z[q,q^{-1}]$ module of rank $|W|$ for these groups.
\end{theorem}

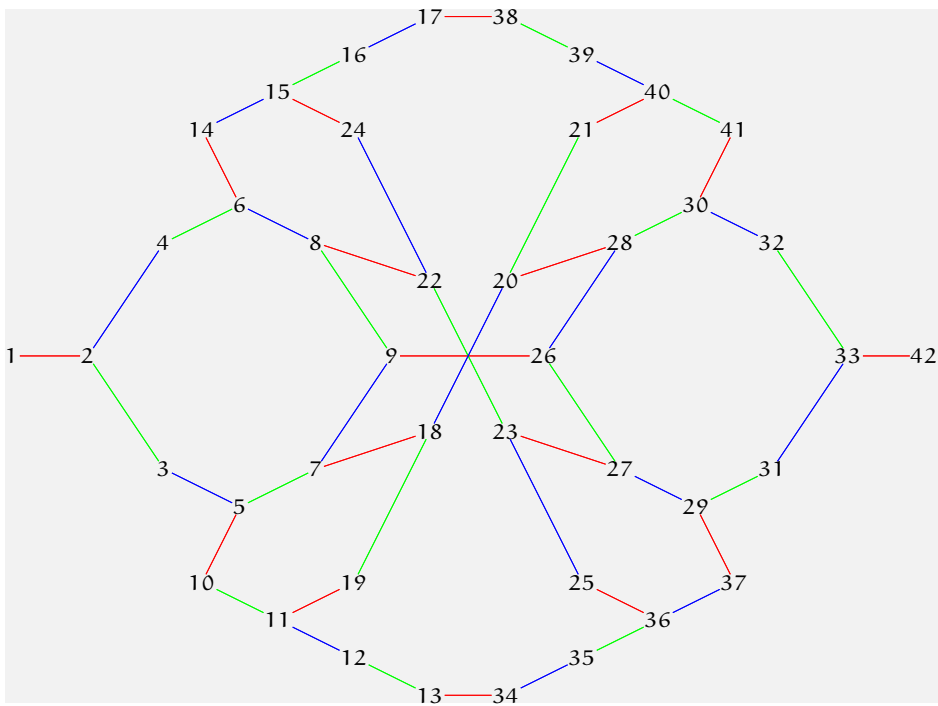
\begin{figure}[ht]
\begin{center}
\begin{tikzpicture}
\draw (0,0) node (1) {$_{1}$};
\draw (1) ++ (1,0) node (2) {$_2$};
\draw (2) ++ (1,-1.5) node (3) {$_3$};
\draw (2) ++ (1,1.5) node (4) {$_4$};
\draw (2) ++ (2,-2) node (5) {$_5$};
\draw (2) ++ (2,2) node (6) {$_6$};
\draw (2) ++ (3,-1.5) node (7) {$_7$};
\draw (2) ++ (3,1.5) node (8) {$_8$};
\draw (2) ++ (4,0) node (9) {$_9$};

\draw (9) ++ (2,0) node (26) {$_{26}$};
\draw (26) ++ (1,-1.5) node (27) {$_{27}$};
\draw (26) ++ (1,1.5) node (28) {$_{28}$};
\draw (26) ++ (2,-2) node (29) {$_{29}$};
\draw (26) ++ (2,2) node (30) {$_{30}$};
\draw (26) ++ (3,-1.5) node (31) {$_{31}$};
\draw (26) ++ (3,1.5) node (32) {$_{32}$};
\draw (26) ++ (4,0) node (33) {$_{33}$};

\draw (33) ++ (1,0) node (42) {$_{42}$};

\draw (5) ++ (-0.5,-1) node (10) {$_{10}$};
\draw (10) ++ (1,-0.5) node (11) {$_{11}$};
\draw (11) ++ (1,-0.5) node (12) {$_{12}$};
\draw (12) ++ (1,-0.5) node (13) {$_{13}$};

\draw (6) ++ (-0.5,1) node (14) {$_{14}$};
\draw (14) ++ (1,0.5) node (15) {$_{15}$};
\draw (15) ++ (1,0.5) node (16) {$_{16}$};
\draw (16) ++ (1,0.5) node (17) {$_{17}$};

\draw (9) ++ (0.5,-1) node (18) {$_{18}$};
\draw (18) ++ (-1,-2) node(19) {$_{19}$};
\draw (26) ++ (-0.5,1) node (20) {$_{20}$};
\draw (20) ++ (1,2) node (21) {$_{21}$};

\draw (13) ++ (1,0) node (34) {$_{34}$};
\draw (34) ++ (1,0.5) node (35) {$_{35}$};
\draw (35) ++ (1,0.5) node (36) {$_{36}$};
\draw (36) ++ (1,0.5) node (37) {$_{37}$};

\draw (17) ++ (1,0) node (38) {$_{38}$};
\draw (38) ++ (1,-0.5) node (39) {$_{39}$};
\draw (39) ++ (1,-0.5) node (40) {$_{40}$};
\draw (40) ++ (1,-0.5) node (41) {$_{41}$};

\draw (18) ++ (1,0) node (23) {$_{23}$};
\draw (23) ++ (1,-2) node (25) {$_{25}$};
\draw (23) ++ (-1,2) node (22) {$_{22}$};
\draw (22) ++ (-1,2) node (24) {$_{24}$};

%%  1
\draw[r] (1) -- (2);
\draw[r] (2) -- (1);
\draw[r] (5) -- (10);
\draw[r] (6) -- (14);
\draw[r] (7) -- (18);
\draw[r] (8) -- (22);
\draw[r] (9) -- (26);
\draw[r] (10) -- (5);
\draw[r] (11) -- (19);
\draw[r] (13) -- (34);
\draw[r] (14) -- (6);
\draw[r] (15) -- (24);
\draw[r] (17) -- (38);
\draw[r] (18) -- (7);
\draw[r] (19) -- (11);
\draw[r] (20) -- (28);
\draw[r] (21) -- (40);
\draw[r] (22) -- (8);
\draw[r] (23) -- (27);
\draw[r] (24) -- (15);
\draw[r] (25) -- (36);
\draw[r] (26) -- (9);
\draw[r] (27) -- (23);
\draw[r] (28) -- (20);
\draw[r] (29) -- (37);
\draw[r] (30) -- (41);
\draw[r] (33) -- (42);
\draw[r] (34) -- (13);
\draw[r] (36) -- (25);
\draw[r] (37) -- (29);
\draw[r] (38) -- (17);
\draw[r] (40) -- (21);
\draw[r] (41) -- (30);
\draw[r] (42) -- (33);

\draw[g] (2) -- (3);
\draw[g] (3) -- (2);
\draw[g] (4) -- (6);
\draw[g] (5) -- (7);
\draw[g] (6) -- (4);
\draw[g] (7) -- (5);
\draw[g] (8) -- (9);
\draw[g] (9) -- (8);
\draw[g] (10) -- (11);
\draw[g] (11) -- (10);
\draw[g] (12) -- (13);
\draw[g] (13) -- (12);
\draw[g] (15) -- (16);
\draw[g] (16) -- (15);
\draw[g] (18) -- (19);
\draw[g] (19) -- (18);
\draw[g] (20) -- (21);
\draw[g] (21) -- (20);
\draw[g] (22) -- (23);
\draw[g] (23) -- (22);
\draw[g] (26) -- (27);
\draw[g] (27) -- (26);
\draw[g] (28) -- (30);
\draw[g] (29) -- (31);
\draw[g] (30) -- (28);
\draw[g] (31) -- (29);
\draw[g] (32) -- (33);
\draw[g] (33) -- (32);
\draw[g] (35) -- (36);
\draw[g] (36) -- (35);
\draw[g] (38) -- (39);
\draw[g] (39) -- (38);
\draw[g] (40) -- (41);
\draw[g] (41) -- (40);

\draw[b] (2) -- (4);
\draw[b] (3) -- (5);
\draw[b] (4) -- (2);
\draw[b] (5) -- (3);
\draw[b] (6) -- (8);
\draw[b] (7) -- (9);
\draw[b] (8) -- (6);
\draw[b] (9) -- (7);
\draw[b] (11) -- (12);
\draw[b] (12) -- (11);
\draw[b] (14) -- (15);
\draw[b] (15) -- (14);
\draw[b] (16) -- (17);
\draw[b] (17) -- (16);
\draw[b] (18) -- (20);
\draw[b] (20) -- (18);
\draw[b] (22) -- (24);
\draw[b] (23) -- (25);
\draw[b] (24) -- (22);
\draw[b] (25) -- (23);
\draw[b] (26) -- (28);
\draw[b] (27) -- (29);
\draw[b] (28) -- (26);
\draw[b] (29) -- (27);
\draw[b] (30) -- (32);
\draw[b] (31) -- (33);
\draw[b] (32) -- (30);
\draw[b] (33) -- (31);
\draw[b] (34) -- (35);
\draw[b] (35) -- (34);
\draw[b] (36) -- (37);
\draw[b] (37) -- (36);
\draw[b] (39) -- (40);
\draw[b] (40) -- (39);
\begin{pgfonlayer}{background}
\node [fill=black!5,fit=(current bounding box.north west) (current bounding box.south east)] {};
\end{pgfonlayer}
\end{tikzpicture}
\end{center}
\caption{The coset graph for $G_{24}$}
\label{figcosetgraphG24}
\end{figure}

In Shephard and Todd notation, this statement covers the groups $G_{12}$, $G_{22}$, $G_{24}$, $G_{27}$, $G_{29}$, $G_{31}$, $G_{33}$ and $G_{34}$.

Together with previous results, this theorem admits several corollaries. We refer to \cite{CYCLO} or \cite{BMR} for the general
statement of the BMR freeness conjecture we are referring to in these corollaries. We also recall that this conjecture was already
known to hold for the general series of imprimitive complex reflection groups when it was stated (see \cite{BMR}, theorem 4.24). Therefore, one only
needs to focus on primitive, exceptional reflection groups.

First of all, it has been recently proved by E. Chavli in her thesis \cite{CHAVLI} that the group $G_{13}$, which is generated by 2-reflections but has two reflection classes, satisfies the conjecture. 
This group is the only primitive 2-reflection group having more than one reflection class.
Therefore, we get the following corollary.
\begin{cor} 
Every irreducible complex 2-reflection group satisfies the freeness conjecture.
\end{cor}

It has been proved in \cite{CUBIC} and \cite{CYCLO} that the groups $G_{25}$, $G_{26}$ and $G_{32}$ satisfy the freeness conjecture.
In addition, Etingof and Rains have proved in \cite{ETINGOFRAINS} that the groups of rank 2 satisfy the \emph{weak} freeness conjecture, namely that
their Hecke algebra is finitely generated (and therefore has the right dimension as vector space over the field of fractions of
the generic coefficients) -- see again \cite{CYCLO} for further details, and
see also the recent paper \cite{LOSEV} for more implications. As a consequence, we get the following corollary.

\begin{cor} 
Every irreducible complex reflection group satisfies the \emph{weak} freeness conjecture.
\end{cor}

In order to prove the theorem, we need a presentation of the braid groups. For groups of rank 2,
presentations were first obtained by Bannai, in \cite{BANNAI}. For groups of higher rank,
using the Zariski-Van Kampen method for computing presentations of fundamental groups, a conjectural presentation of $B$ 
 was found by empirical means by Bessis and Michel in \cite{BESSISMICHEL}. The proof that these presentations
were correct did depend on the verification of a geometric criterion. This justification
was subsequently provided in \cite{BESSIS}. Moreover, one finds in \cite{BESSIS} another way to justify these
presentations in the case of well-generated groups, that is, when the minimal number of reflections needed to generate $W$ is equal to the rank of $W$
-- this is the case for all the 2-reflection groups of higher rank except $G_{31}$.
Note however that, because of Proposition~\ref{propHeckeBMR}, we do not really need a \emph{presentation} of $B$, but
only to know that the chosen generators are braided reflections, and that the relations we use are valid -- but we do not really need to
check that they are sufficient to define the group.

From such a presentation, we can describe $H$ as the $\Z[q,q^{-1}]$-algebra defined by the same generators $s_i$ submitted to the defining
relations of the group together with
the additional relations $s_i^2 = (q-1)s_i + q$.
Indeed, it can be shown (see \cite{BMR}) that all the braided reflections
are conjugated to one another as soon as $W$ admits a single reflection class;
therefore, every relation $s^2 = (q-1)s + q$ for $s$ a braided reflection
is implied by the single relation $s_1^2 = (q-1)s_1 + q$.

In order to prove the theorem, we use the following lemma, for which we do not know any proof that does not rely on the
classification.

\begin{Lemma} \label{lem;parab} Every  irreducible complex 2-reflection group $W$ has a maximal parabolic subgroup which is a
Coxeter group.
\end{Lemma}
\begin{proof} If $W$ belongs to the infinite series of complex reflection groups, of type $G(de,e,n)$ in Shephard and Todd notation,
the subgroup $G(1,1,n)$ of permutation matrices, which is a Coxeter group of type $A_{n-1}$, is a maximal parabolic subgroup,
except when $G(de,e,n) = G(1,1,n)$ is itself a Coxeter group. If $W$ is an exceptional group of 2-reflections of rank 2, the subgroup
generated by either of its reflection is a maximal parabolic subgroup of Coxeter type $A_1$. In higher rank the groups $G_{24}$
and $G_{27}$ admit a maximal parabolic subgroup of Coxeter type $B_2$, and the groups $G_{29}$, $G_{31}$, $G_{33}$ and $G_{34}$
admit maximal parabolic subgroups of Coxeter types $B_3$, $A_3$, $A_4$, $D_5$ respectively.
\end{proof}

We then prove the theorem as follows. We know by \cite{BMR} that to any such maximal parabolic subgroup $W_0$ is attached a (non-canonical)
embedding $B_0 \to B$ of the braid groups of $W_0$ inside $B$. Among the presentations of \cite{BESSISMICHEL}, we choose
one for which such an embedding corresponds to the choice of a proper subset $I$ of the set of indices involved in the
presentation of $B$. That is, we can identify $B_0$ with the subgroup of $B$ generated by the corresponding generators,
and defining relations of $B_0$ are given by the set of all the relations of the given presentation of $B$ which do not involve any
generator of $B_0$. In rank at least 3, the relations of Coxeter type in these presentations can be depicted inside a Coxeter-like diagram,
see Figure~\ref{fig:diagrams}. Of course, there are additional relations involving 3 generators, that we will describe in due time
(for $G_{31}$, these are represented by a circle in the diagram).

This group homomorphism $B_0 \to B$ induces an algebra morphism $H_0 \to H$, where $H_0$ denotes the (usual) Hecke algebra of $W_0$. Although we
do not  know {\it a priori} that it is injective, it nevertheless endows $H$ with the structure of a $H_0$-module.

We prove the following, for $W$ an irreducible complex 2-reflection group with a single reflection
class but the largest one $G_{34}$, and $W_0$ the parabolic subgroup provided by Lemma~\ref{lem;parab}. 

\begin{prop} \label{propGeneH0} As a $H_0$-module, $H$ is generated by $|W/W_0|$ elements.
\end{prop}

By the classical theory of Iwahori-Hecke algebras we know that $H_0$ is generated as a $\Z[q,q^{-1}]$-module by $|W_0|$ elements ;
therefore Proposition~\ref{propGeneH0} implies that $H$ is generated as a $\Z[q,q^{-1}]$-module
by $|W| = |W_0|.|W/W_0|$ elements and Proposition~\ref{propHeckeBMR} finally implies the theorem for all groups but $G_{34}$.

Once it is proved, the theorem implies that the map $H_0 \to H$ is indeed injective. Actually, Propositions~\ref{propGeneH0} and~\ref{propHeckeBMR} together imply a statement a bit stronger than the theorem (for all groups but $G_{34}$), namely:
\begin{prop} \label{propFreeH0} As a $H_0$-module, $H$ is a free $H_0$-module of rank $|W/W_0|$.
\end{prop}

We now explain how we prove Proposition~\ref{propGeneH0}.
We denote by $S$ the set of generators $s_i$ of~$W$.
In each case, we choose a system of representatives of $W/W_0$, and more specifically a set
$x_l, l \in \{1, \dots, |W/W_0| \}$, of words in the $s_i$ of minimal length whose images in $W$ represent all the classes of $W/W_0$.
We show that the $H_0$-submodule $\sum_l H_0 x_l$ is a right ideal in $H$. Since it contains the identity of $H$ this will prove our Proposition~\ref{propGeneH0}.
For this we need to establish $|W/W_0|\,|S|$ relations of the form $x_l . s = \sum_{1 \leq k \leq |W:W_0|} \alpha_{l,k}(s) x_k$
with $\alpha_{l,k}(s) \in H_0$. This is basically what we do.

\begin{figure}
\begin{tikzpicture}[dba/.style={double,double equal sign distance}]
\draw (0,0) node (10) {$2$};
\draw (10)++(1,0) node (11) {$3$};
\draw (10)++(.5,1) node (12) {$1$};
\draw (10) edge[dba] (11);
\draw (10)--(12);
\draw (11)--(12);
\draw (11)++(1,0) node (6) {$1$};
\draw (6)++(1,0) node (7) {$2$};
\draw (7)++(1,0) node (8) {$3$};
\draw (7)++(.5,1) node (9) {$4$};
\draw[u] (6)--(7);
\draw[u] (7)--(9);
\draw[u] (8)--(9);
\draw[u] (7) edge[dba]  (8);
\draw (8)++(1,1) node (19) {$4$};
\draw (8)++(2,0) node (20) {$1$};
\draw (20)++(1,0) node (21) {$3$};
\draw (19)++(1.5,0) node (22) {$2$};
\draw (22)++(1.5,0) node (23) {$5$};
\draw[u]  (19)--(22);
\draw[u]  (19)--(20);
\draw[u]  (22)--(23);
\draw[u]  (21)--(23);
\draw (20)++(.5,.35) circle (12pt);
\draw (8)++(5,0) node (13) {$1$}; 
\draw (13)++(1,0) node (14) {$2$};
\draw (14)++(1,0) node (15) {$4$};
\draw (15)++(1,0) node (16) {$5$};
\draw (14)++(.5,1) node (17) {$3$};
\draw[u] (13)--(14);
\draw[u] (14)--(15);
\draw[u] (15)--(16);
\draw[u] (14)--(17);
\draw[u] (15)--(17);
\begin{pgfonlayer}{background}
\node [fill=black!5,fit=(current bounding box.north west) (current bounding box.south east)] {};
\end{pgfonlayer}
\end{tikzpicture}
\caption{Diagrammatic presentations for the Coxeter relations of the groups $G_{24}/G_{27}$, $G_{29}$, $G_{31}$, $G_{33}$}
\label{fig:diagrams}
\end{figure}
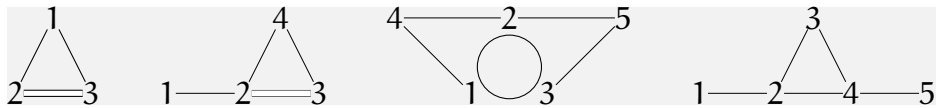

In Section~\ref{sec:G24} we will prove the conjecture for the group $G_{24}$ following this procedure `by hand'
by establishing a number of equations
of the form $m.s = \dots$ for $m$ some word in the generators. This involves a well-defined
ordering in the building of coset representatives, plus a well-defined ordering of the entries that we fill in, so that the
computation of each entry does not involve entries that are not yet filled in. A visual support is given by the
`coset graph' for $W/W_0$, namely the graph whose vertices are the (images in $W_0$ of the) $x_l$,
and an edge $x_l \to^s x_n$ means that $x_n$ is defined as $x_l.s$. The graph for $G_{24}$ is given by Figure~\ref{figcosetgraphG24},
the three different colors for the edges corresponding to the 3 generators of the group.
The graphs for $G_{12}$ and $G_{29}$ are similarly depicted in Figures~\ref{figcosetgraphG12} and~\ref{figcosetgraphG29}.

\begin{figure}[ht]
\begin{center}
\begin{tikzpicture}
\draw (0,0) node (1) {$_{1}$}; % 0
\draw (1) ++ (1,-0.5) node (2) {$_2$};  % 2
\draw (2) ++ (1.5,1) node (3) {$_3$};  % 21
\draw (1) ++ (1,0.5) node (4) {$_4$};  % 3
\draw (4) ++ (1.5,-1) node (5) {$_5$};  % 31
\draw (2) ++ (0.5,-1) node (6) {$_6$};  % 23
\draw (6) ++ (0,-1) node (7) {$_7$};  % 231
\draw (3) ++ (2,-1) node (8) {$_8$};  % 212
\draw (8) ++ (1.5,1) node (9) {$_9$};  % 2121
\draw (3) ++ (0.5,1.5) node (10) {$_{10}$};  % 213
\draw (10) ++ (1,0) node (11) {$_{11}$};  % 2131
\draw (4) ++ (0.5,1) node (12) {$_{12}$};  % 32
\draw (12) ++ (0,1) node (13) {$_{13}$};  % 321
\draw (5) ++ (0.5,-1.5) node (14) {$_{14}$};  % 312
\draw (14) ++ (1,0) node (15) {$_{15}$};  % 3121
\draw (5) ++ (2,1) node (16) {$_{16}$};  % 313
\draw (16) ++ (1.5,-1) node (17) {$_{17}$};  % 3131
\draw (6) ++ (2,0.5) node (18) {$_{18}$};  % 232
\draw (18) ++ (0,2) node (19) {$_{19}$};  % 2321
\draw (9) ++ (1,-0.5) node (20) {$_{20}$};  % 21212
\draw (9) ++ (-0.5,1) node (21) {$_{21}$};  % 21213
\draw (21) ++ (0,1) node (22) {$_{22}$};  % 212131
\draw (15) ++ (1.5,-0.5) node (23) {$_{23}$};  % 31212
\draw (23) ++ (0,1) node (24) {$_{24}$};  % 312121

% 1 
\draw[r] (2) -- (3);
\draw[r] (4) -- (5);
\draw[r] (6) -- (7);
\draw[r] (8) -- (9);
\draw[r] (10) -- (11);
\draw[r] (12) -- (13);
\draw[r] (14) -- (15);
\draw[r] (16) -- (17);
\draw[r] (18) -- (19);
\draw[r] (21) -- (22);
\draw[r] (23) -- (24);

% 2
\draw[g] (1) -- (2);
\draw[g] (3) -- (8);
\draw[g] (4) -- (12);
\draw[g] (5) -- (14);
\draw[g] (6) -- (18);
\draw[g] (9) -- (20);
\draw[g] (10) -- (13);
\draw[g] (11) -- (16);
\draw[g] (15) -- (23);
\draw[g] (17) -- (24);
\draw[g] (19) -- (21);

% 3
\draw[b] (1) -- (4);
\draw[b] (2) -- (6);
\draw[b] (3) -- (10);
\draw[b] (5) -- (16);
\draw[b] (7) -- (14);
\draw[b] (8) -- (15);
\draw[b] (9) -- (21);
\draw[b] (11) -- (22);
\draw[b] (12) -- (19);
\draw[b] (17) -- (20);
\draw[b] (18) -- (24);

\begin{pgfonlayer}{background}
\node [fill=black!5,fit=(current bounding box.north west) (current bounding box.south east)] {};
\end{pgfonlayer}
\end{tikzpicture}
\end{center}
\caption{The coset graph for $G_{12}$}
\label{figcosetgraphG12}
\end{figure}
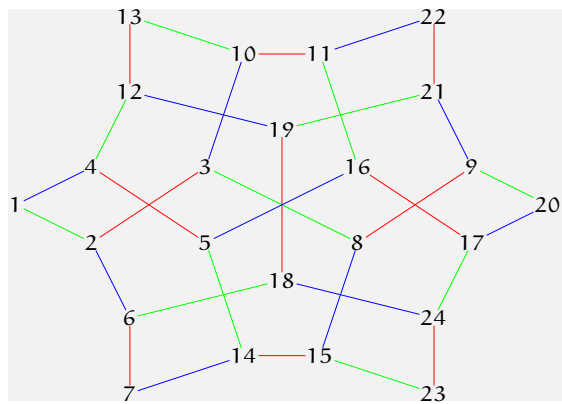

\begin{figure}\scriptsize
\begin{tikzpicture}[scale=0.8]
\draw (0,0) node (1) {$_{1}$};
\draw (1) ++ (1,0) node (2) {$_2$};
\draw (2) ++ (1,1.5) node (3) {$_3$};
\draw (2) ++ (1,-1.5) node (4) {$_4$};
\draw (2) ++ (2,3) node (5) {$_5$};
\draw (2) ++ (2,2) node (6) {$_6$};
\draw (2) ++ (2,-2) node (7) {$_7$};
\draw (2) ++ (3,3.5) node (8) {$_8$};
\draw (2) ++ (3,1.5) node (9) {$_9$};
\draw (6) ++ (-0.5,1) node (10) {$_{26}$};
\draw (2) ++ (3,-2.5) node (11) {$_{10}$};
\draw (2) ++ (3,-1.5) node (12) {$_{11}$};
\draw (7) ++ (0.5,-1.5) node (13) {$_{32}$};
\draw (2) ++ (4,3.5) node (14) {$_{12}$};
\draw (10) ++ (0.5,1) node (15) {$_{27}$};
\draw (2) ++ (4,1.5) node (16) {$_{13}$};
\draw (2) ++ (4,0) node (17) {$_{14}$};
\draw (2) ++ (4,-2) node (18) {$_{15}$};
\draw (13) ++ (1,-0.5) node (19) {$_{33}$};
\draw (2) ++ (5,3) node (20) {$_{16}$};
\draw (2) ++ (5,4) node (21) {$_{17}$};
\draw (14) ++ (1,0) node (22) {$_{38}$};
\draw (10) ++ (1,2) node (23) {$_{28}$};
\draw (2) ++ (5,0) node (24) {$_{18}$};
\draw (2) ++ (4,-1) node (25) {$_{50}$};
\draw (2) ++ (5,-2) node (26) {$_{19}$};
\draw (13) ++ (2,-1) node (27) {$_{34}$};
\draw (2) ++ (6,3.5) node (28) {$_{20}$};
\draw (20) ++ (1.5,-0.5) node (29) {$_{39}$};
\draw (2) ++ (5.5,5.5) node (30) {$_{56}$};
\draw (23) ++ (1,-1) node (31) {$_{40}$};

\draw (30) ++ (0.5,-1) node (43) {$_{59}$};
\draw (43) ++ (0.5,-1.5) node (32) {$_{41}$};

\draw (10) ++ (1.5,3) node (33) {$_{29}$};
\draw (2) ++ (6,-0.5) node (34) {$_{21}$};
\draw (25) ++ (2,1) node (35) {$_{51}$};
\draw (2) ++ (6,-1.5) node (36) {$_{22}$};
\draw (2) ++ (6.5,-2.5) node (37) {$_{80}$};
\draw (13) ++ (3,-1.5) node (38) {$_{35}$};
\draw (2) ++ (7,3) node (39) {$_{23}$};
\draw (30) ++ (1.5,-1) node (40) {$_{57}$};
%HERE
\draw (30) ++ (2,-2) node (56) {$_{61}$};
\draw (56) ++ (0.5,-1.5) node (41) {$_{42}$};
\draw (30) ++ (0,1) node (42) {$_{58}$};
% 43 done
\draw (41) ++ (0,0.5) node (44) {$_{43}$};
\draw (10) ++ (2,4) node (45) {$_{30}$};
\draw (2) ++ (7,0) node (46) {$_{24}$};
\draw (37) ++ (0.5,1) node (47) {$_{81}$};
\draw (25) ++ (4,2) node (48) {$_{52}$};
\draw (36) ++ (1.5,0.5) node (49) {$_{104}$};
\draw (37) ++ (0.5,-1) node (50) {$_{82}$};
\draw (13) ++ (4,-2) node (51) {$_{36}$};
\draw (38) ++ (1,1) node (52) {$_{111}$};
\draw (2) ++ (8,1.5) node (53) {$_{25}$};

% put 54 into context
\draw (30) ++ (1,1) node (91) {$_{119}$};
\draw (91) ++ (1,-2) node (71) {$_{117}$};
\draw (71) ++ (1,-2) node (54) {$_{116}$};
\draw (54) ++ (1,-2) node (72) {$_{118}$};
\draw (72) ++ (1,-2) node (92) {$_{120}$};
\draw (92) ++ (1,-2) node (111) {$_{121}$};

\draw (40) ++ (0.5,1) node (55) {$_{60}$};
% 56 done 

% put 57 into context
\draw (37) ++ (4,2) node (78) {$_{90}$};
\draw (78) ++ (-0.5,1.5) node (57) {$_{44}$};

\draw (42) ++ (0.5,1) node (58) {$_{62}$};
\draw (42) ++ (-0.5,1) node (59) {$_{63}$};
\draw (57) ++ (0,-0.5) node (60) {$_{45}$};
\draw (10) ++ (2.5,5) node (61) {$_{31}$};
\draw (46) ++ (1.5,0.5) node (62) {$_{105}$};
\draw (47) ++ (0.5,1) node (63) {$_{83}$};
\draw (47) ++ (0.5,-1) node (64) {$_{84}$};
\draw (25) ++ (6,3) node (65) {$_{53}$};
\draw (62) ++ (-0.5,0) node (66) {$_{106}$};
\draw (50) ++ (0,1) node (67) {$_{107}$};
\draw (50) ++ (0,-1) node (68) {$_{85}$};
\draw (13) ++ (5,-2.5) node (69) {$_{37}$};
\draw (53) ++ (2,0) node (70) {$_{122}$};
% 71 done
% 72 done
\draw (58) ++ (1.5,-1) node (73) {$_{64}$};
\draw (55) ++ (1.5,1) node (74) {$_{65}$};
\draw (56) ++ (1.5,-1) node (75) {$_{66}$};

\draw (78) ++ (1.5,-1) node (96) {$_{94}$};
\draw (96) ++ (-0.5,1.5) node (76) {$_{46}$};

\draw (96) ++ (1,1.5) node (97) {$_{131}$};
\draw (97) ++ (-1.5,0.5) node (77) {$_{47}$};

% 78 done
\draw (59) ++ (0.5,1) node (79) {$_{67}$};

\draw (70) ++ (1,1.5) node (89) {$_{123}$};
\draw (89) ++ (-1.5,-0.5) node (80) {$_{108}$};

\draw (64) ++ (0.5,1.5) node (81) {$_{109}$};

\draw (63) ++ (1.5,1) node (82) {$_{86}$};
\draw (64) ++ (1.5,-1) node (83) {$_{87}$};
\draw (25) ++ (8,4) node (84) {$_{54}$};
\draw (75) ++ (1.5,1) node (85) {$_{70}$};
\draw (80) ++ (0.5,0) node (86) {$_{110}$};
\draw (68) ++ (1.5,-1) node (87) {$_{88}$};
\draw (68) ++ (0.5,-1) node (88) {$_{89}$};
% 89 done
\draw (70) ++ (1,-1.5) node (90) {$_{124}$};
% 91 done
% 92 done
\draw (73) ++ (1.5,1) node (93) {$_{68}$};
\draw (74) ++ (1.5,-1) node (94) {$_{69}$};

\draw (97) ++ (1,-0.5) node (115) {$_{135}$};
\draw (115) ++ (-1,0) node (95) {$_{48}$};

% 96 done
% 97 done
\draw (78) ++ (0.5,-1) node (98) {$_{95}$};
\draw (79) ++ (0.5,1) node (99) {$_{71}$};

\draw (89) ++ (1,1.5) node (108) {$_{125}$};
\draw (108) ++ (-1.5,-0.5) node (100) {$_{112}$};

\draw (94) ++ (-0.5,-1.5) node (101) {$_{113}$};
\draw (83) ++ (0,-1) node (102) {$_{91}$};
\draw (83) ++ (1.5,1) node (103) {$_{92}$};
\draw (25) ++ (10,5) node (104) {$_{55}$};

\draw (89) ++ (1,0.5) node (109) {$_{126}$};
\draw (109) ++ (1,-0.5) node (105) {$_{129}$};

\draw (85) ++ (0.5,1) node (106) {$_{74}$};
\draw (87) ++ (0.5,-1) node (107) {$_{93}$};
% 108 done
% 109 done
\draw (90) ++ (1,-0.5) node (110) {$_{127}$};
% 111 done
\draw (93) ++ (0.5,1) node (112) {$_{72}$};
\draw (94) ++ (0.5,1) node (113) {$_{73}$};

\draw (110) ++ (1,-0.5) node (125) {$_{130}$};
\draw (125) ++ (2.5,-1) node (129) {$_{146}$};
\draw (129) ++ (-1,1) node (114) {$_{49}$};
% 115 done
\draw (98) ++ (1.5,-1) node (116) {$_{98}$};
\draw (97) ++ (1,1.5) node (117) {$_{134}$};
\draw (99) ++ (1.5,1) node (118) {$_{75}$};
\draw (113) ++ (0,-1) node (119) {$_{114}$};
\draw (102) ++ (1.5,1) node (120) {$_{96}$};
\draw (105) ++ (1,0) node (121) {$_{133}$};
\draw (106) ++ (0.5,1) node (122) {$_{78}$};
\draw (107) ++ (1.5,-1) node (123) {$_{97}$};
\draw (108) ++ (1,0.5) node (124) {$_{128}$};
% 125 done
\draw (112) ++ (1.5,-1) node (126) {$_{76}$};
\draw (118) ++ (1.5,-1) node (127) {$_{77}$};
\draw (112) ++ (1,0.5) node (128) {$_{152}$};
% 129 done
\draw (115) ++ (1,0) node (130) {$_{139}$};
\draw (116) ++ (0,-1) node (131) {$_{101}$};
\draw (117) ++ (1,0) node (132) {$_{138}$};
\draw (118) ++ (0,1) node (133) {$_{158}$};

\draw (128) ++ (1,-0.5) node (141) {$_{153}$};
\draw (141) ++ (1,-0.5) node (147) {$_{154}$};
\draw (147) ++ (-1,-1) node (134) {$_{115}$};

\draw (120) ++ (1.5,-1) node (135) {$_{99}$};
\draw (121) ++ (1,1.5) node (136) {$_{136}$};
\draw (123) ++ (1.5,1) node (137) {$_{100}$};
\draw (123) ++ (0,-1) node (138) {$_{159}$};
\draw (124) ++ (1,0) node (139) {$_{132}$};
\draw (127) ++ (1.5,-1) node (140) {$_{79}$};
% 141 done
\draw (129) ++ (0.5,1) node (142) {$_{147}$};
\draw (129) ++ (-0.5,-1) node (143) {$_{148}$};
\draw (132) ++ (1,-0.5) node (144) {$_{141}$};
\draw (130) ++ (1,0.5) node (145) {$_{142}$};
\draw (131) ++ (0.5,-1) node (146) {$_{103}$};
% 147 done
\draw (135) ++ (0.5,-1) node (148) {$_{102}$};
\draw (136) ++ (1,0.5) node (149) {$_{140}$};
\draw (139) ++ (1,0.5) node (150) {$_{137}$};
\draw (142) ++ (0.5,1) node (151) {$_{149}$};
\draw (143) ++ (-0.5,-1) node (152) {$_{150}$};
\draw (144) ++ (1,0.5) node (153) {$_{144}$};
\draw (147) ++ (1,-0.5) node (154) {$_{155}$};
\draw (152) ++ (-0.5,-1) node (155) {$_{151}$};
\draw (149) ++ (1,-0.5) node (156) {$_{143}$};
\draw (154) ++ (1,-0.5) node (158) {$_{156}$};
\draw (158) ++ (1,-0.5) node (157) {$_{157}$};
\draw (156) ++ (1,-1.5) node (159) {$_{145}$};
\draw (159) ++ (1,0) node (160) {$_{160}$};

%%  1
\draw[o] (3) -- (5);
\draw[o] (6) -- (8);
\draw[o] (7) -- (11);
\draw[o] (9) -- (16);
\draw[o] (10) -- (15);
\draw[o] (12) -- (18);
\draw[o] (13) -- (19);
\draw[o] (14) -- (20);
\draw[o] (17) -- (24);
\draw[o] (21) -- (28);
\draw[o] (22) -- (29);
\draw[o] (25) -- (35);
\draw[o] (26) -- (34);
\draw[o] (30) -- (40);
\draw[o] (32) -- (41);
\draw[o] (36) -- (46);
\draw[o] (37) -- (47);
\draw[o] (42) -- (58);
\draw[o] (43) -- (56);
\draw[o] (44) -- (60);
\draw[o] (45) -- (61);
\draw[o] (49) -- (62);
\draw[o] (50) -- (64);
\draw[o] (51) -- (69);
\draw[o] (55) -- (73);
\draw[o] (57) -- (76);
\draw[o] (59) -- (79);
\draw[o] (66) -- (86);
\draw[o] (67) -- (81);
\draw[o] (68) -- (87);
\draw[o] (71) -- (91);
\draw[o] (74) -- (93);
\draw[o] (77) -- (95);
\draw[o] (78) -- (96);
\draw[o] (80) -- (100);
\draw[o] (83) -- (102);
\draw[o] (84) -- (104);
\draw[o] (88) -- (107);
\draw[o] (89) -- (108);
\draw[o] (92) -- (111);
\draw[o] (94) -- (113);
\draw[o] (97) -- (115);
\draw[o] (98) -- (116);
\draw[o] (101) -- (119);
\draw[o] (103) -- (120);
\draw[o] (105) -- (121);
\draw[o] (106) -- (122);
\draw[o] (109) -- (124);
\draw[o] (110) -- (125);
\draw[o] (112) -- (126);
\draw[o] (117) -- (132);
\draw[o] (127) -- (140);
\draw[o] (128) -- (141);
\draw[o] (130) -- (144);
\draw[o] (131) -- (135);
\draw[o] (136) -- (139);
\draw[o] (142) -- (151);
\draw[o] (145) -- (153);
\draw[o] (146) -- (148);
\draw[o] (149) -- (150);
\draw[o] (152) -- (155);
\draw[o] (157) -- (158);

%% 2
\draw[g] (2) -- (3);
\draw[g] (4) -- (7);
\draw[g] (6) -- (9);
\draw[g] (8) -- (14);
\draw[g] (12) -- (17);
\draw[g] (15) -- (23);
\draw[g] (16) -- (20);
\draw[g] (18) -- (26);
\draw[g] (19) -- (27);
\draw[g] (22) -- (31);
\draw[g] (24) -- (34);
\draw[g] (28) -- (39);
\draw[g] (30) -- (42);
\draw[g] (32) -- (44);
\draw[g] (33) -- (45);
\draw[g] (35) -- (48);
\draw[g] (38) -- (51);
\draw[g] (40) -- (55);
\draw[g] (41) -- (57);
\draw[g] (46) -- (53);
\draw[g] (47) -- (63);
\draw[g] (49) -- (66);
\draw[g] (50) -- (68);
\draw[g] (52) -- (67);
\draw[g] (54) -- (71);
\draw[g] (56) -- (75);
\draw[g] (58) -- (73);
\draw[g] (60) -- (76);
\draw[g] (62) -- (80);
\draw[g] (64) -- (83);
\draw[g] (65) -- (84);
\draw[g] (70) -- (89);
\draw[g] (72) -- (92);
\draw[g] (74) -- (94);
\draw[g] (78) -- (82);
\draw[g] (79) -- (99);
\draw[g] (85) -- (106);
\draw[g] (86) -- (100);
\draw[g] (87) -- (102);
\draw[g] (90) -- (110);
\draw[g] (93) -- (112);
\draw[g] (95) -- (114);
\draw[g] (97) -- (117);
\draw[g] (98) -- (103);
\draw[g] (105) -- (109);
\draw[g] (107) -- (123);
\draw[g] (113) -- (126);
\draw[g] (115) -- (130);
\draw[g] (116) -- (131);
\draw[g] (118) -- (127);
\draw[g] (119) -- (134);
\draw[g] (120) -- (135);
\draw[g] (121) -- (136);
\draw[g] (124) -- (139);
\draw[g] (129) -- (142);
\draw[g] (132) -- (144);
\draw[g] (137) -- (148);
\draw[g] (141) -- (147);
\draw[g] (143) -- (152);
\draw[g] (149) -- (156);
\draw[g] (153) -- (159);
\draw[g] (154) -- (158);

%% 3
\draw[b] (2) -- (4);
\draw[b] (3) -- (6);
\draw[b] (5) -- (8);
\draw[b] (7) -- (12);
\draw[b] (9) -- (17);
\draw[b] (11) -- (18);
\draw[b] (14) -- (21);
\draw[b] (16) -- (24);
\draw[b] (20) -- (28);
\draw[b] (22) -- (32);
\draw[b] (23) -- (33);
\draw[b] (26) -- (36);
\draw[b] (27) -- (38);
\draw[b] (29) -- (41);
\draw[b] (30) -- (43);
\draw[b] (34) -- (46);
\draw[b] (37) -- (50);
\draw[b] (39) -- (53);
\draw[b] (40) -- (56);
\draw[b] (42) -- (59);
\draw[b] (47) -- (64);
\draw[b] (48) -- (65);
\draw[b] (49) -- (67);
\draw[b] (54) -- (72);
\draw[b] (55) -- (74);
\draw[b] (57) -- (77);
\draw[b] (58) -- (79);
\draw[b] (62) -- (81);
\draw[b] (63) -- (82);
\draw[b] (68) -- (88);
\draw[b] (70) -- (90);
\draw[b] (73) -- (93);
\draw[b] (75) -- (85);
\draw[b] (76) -- (95);
\draw[b] (78) -- (98);
\draw[b] (80) -- (101);
\draw[b] (83) -- (103);
\draw[b] (87) -- (107);
\draw[b] (89) -- (109);
\draw[b] (94) -- (106);
\draw[b] (96) -- (116);
\draw[b] (97) -- (110);
\draw[b] (99) -- (118);
\draw[b] (100) -- (119);
\draw[b] (102) -- (120);
\draw[b] (105) -- (117);
\draw[b] (108) -- (124);
\draw[b] (112) -- (127);
\draw[b] (113) -- (122);
\draw[b] (115) -- (125);
\draw[b] (121) -- (132);
\draw[b] (123) -- (137);
\draw[b] (126) -- (140);
\draw[b] (129) -- (143);
\draw[b] (130) -- (145);
\draw[b] (131) -- (146);
\draw[b] (135) -- (148);
\draw[b] (136) -- (149);
\draw[b] (139) -- (150);
\draw[b] (144) -- (153);
\draw[b] (147) -- (154);
\draw[b] (156) -- (159);

%%  4
\draw[r] (1) -- (2);
\draw[r] (6) -- (10);
\draw[r] (7) -- (13);
\draw[r] (8) -- (15);
\draw[r] (11) -- (19);
\draw[r] (14) -- (22);
\draw[r] (17) -- (25);
\draw[r] (20) -- (29);
\draw[r] (21) -- (30);
\draw[r] (23) -- (31);
\draw[r] (24) -- (35);
\draw[r] (26) -- (37);
\draw[r] (28) -- (40);
\draw[r] (32) -- (43);
\draw[r] (34) -- (47);
\draw[r] (36) -- (49);
\draw[r] (38) -- (52);
\draw[r] (39) -- (54);
\draw[r] (41) -- (56);
\draw[r] (45) -- (59);
\draw[r] (46) -- (62);
\draw[r] (48) -- (63);
\draw[r] (50) -- (67);
\draw[r] (51) -- (68);
\draw[r] (53) -- (70);
\draw[r] (55) -- (71);
\draw[r] (57) -- (78);
\draw[r] (61) -- (79);
\draw[r] (64) -- (81);
\draw[r] (65) -- (85);
\draw[r] (69) -- (87);
\draw[r] (72) -- (90);
\draw[r] (73) -- (91);
\draw[r] (75) -- (82);
\draw[r] (76) -- (96);
\draw[r] (77) -- (97);
\draw[r] (80) -- (89);
\draw[r] (84) -- (105);
\draw[r] (92) -- (103);
\draw[r] (94) -- (101);
\draw[r] (95) -- (115);
\draw[r] (98) -- (110);
\draw[r] (100) -- (108);
\draw[r] (104) -- (121);
\draw[r] (106) -- (109);
\draw[r] (111) -- (120);
\draw[r] (112) -- (128);
\draw[r] (113) -- (119);
\draw[r] (114) -- (129);
\draw[r] (116) -- (125);
\draw[r] (118) -- (133);
\draw[r] (122) -- (124);
\draw[r] (123) -- (138);
\draw[r] (126) -- (141);
\draw[r] (130) -- (142);
\draw[r] (134) -- (147);
\draw[r] (144) -- (151);
\draw[r] (146) -- (152);
\draw[r] (148) -- (155);
\draw[r] (149) -- (157);
\draw[r] (150) -- (158);
\draw[r] (159) -- (160);

\begin{pgfonlayer}{background}
\node [fill=black!5,fit=(current bounding box.north west) (current bounding box.south east),scale=0.8] {};
\end{pgfonlayer}
\end{tikzpicture}
\caption{The coset graph for $G_{29}/B_3$}
\label{figcosetgraphG29}
\end{figure}
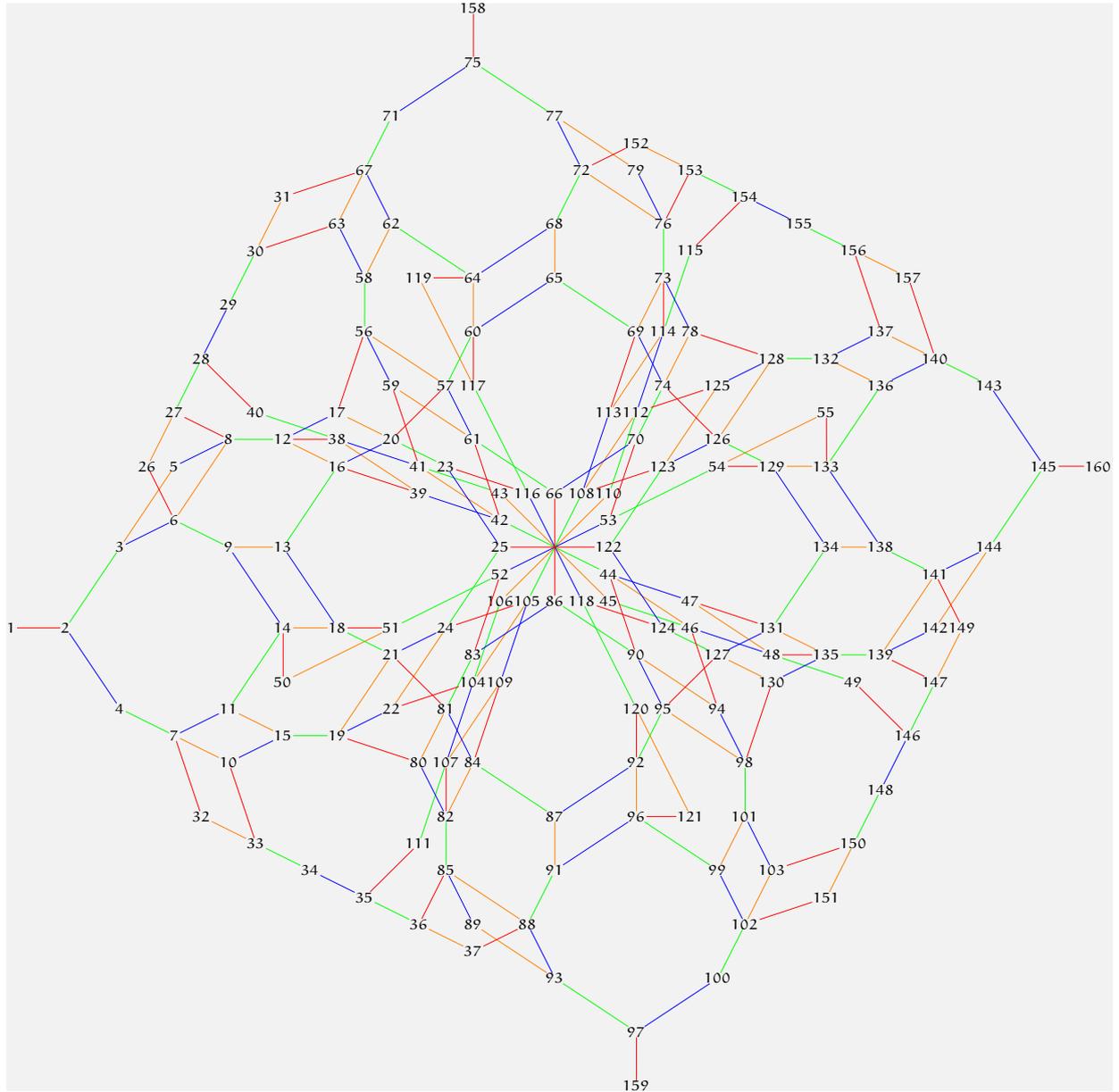

Then, in Section~\ref{sec:algo}, we will show that the procedure
can be automatized : we define algorithms which happen to converge in each case. These algorithms need to know in advance some additional
relations inside $B$, that we found heuristically. The search for such relations and their justification rely on the solution
of the word problem inside $B$. Fortunately, thanks to previous works, all these groups have decidable word problem, and there
are effective software to deal with them ; we explain all this, along with some basic algorithmic procedures, in Section~\ref{sec:genauto}.

Finally, we prove in Section~\ref{sec:proof} the conjecture for all groups, using
the algorithms described above. At the end of this section, we explain
the troubles we got into when dealing with the largest group $G_{34}$,
and the solutions we found to tackle this case, too.

\section{General automatic procedures}
\label{sec:genauto}

We now explain a few tools that we use in a systematic way and for which we will not
detail the calculations.

\subsection{Determining the coset graph}
\label{sec:coset-graph}
The coset graph of $W_0$ in $W$ with respect to $S$ is the graph which has the (right) cosets
$W_0 w$, $w \in W$, as its vertices, and edges $x \stackrel{s}{\mbox{-----}} y$
if $x.s = y$ for cosets $x, y$ and $s \in S$.

The coset graph, together with a distinguished spanning tree, is
determined by a standard orbit algorithm which works on an \emph{ordered} copy
$\hat{S}$ of the set $S$ of generators of $W$,
which induces a fixed order on all subsets $J$  of $S$.

Input: $W$, $\hat{S}$ and a subset $J$ of $S$ generating $W_0$. \\
Output: The coset graph $\Gamma = (V, E)$ of $W_0$ in $W$ with respect
to $S$ and a spanning tree $T \subseteq E$.

1. Initialize a empty queue $Q$, a vertex list $V$ and two edge lists
$E$ and $T$ as empty lists. Then push the
trivial coset $W_0 = \Span{J}$ onto  $Q$ and add it to $V$. \\
2. while $Q$ is not empty: \\
3. \qquad pop the next coset $x$ off $Q$ \\
4. \qquad for  $s \in \hat{S}$: process $(x, s)$. \\
5. return the graph $\Gamma = (V, E)$ and spanning edges $T$.\\

Processing $(x, s)$ is done as follows:

1.    Compute the coset $z:= x.s$ and
add the edge $x \stackrel{s}{\mbox{-----}} z$ to $E$ if not already present.\\
2.    If $z \notin V$: 
 push $z$ onto $Q$ and $V$, and add the edge
 $x \stackrel{s}{\mbox{\bf-----}} z$ to the spanning tree $T$.

Note that the spanning tree $T$ defines, for each coset
$x$, a word $w$ of minimal length in the generators $S$,
which represents the coset when evaluated in $W$.
This word depends on the ordering of $\hat{S}$.
The cosets are enumerated in the lexicographic order induced by
$\hat{S}$ on the set of words in $S$.

It is possible, to group the cosets  into double cosets of $W_0$ in $W$
and to ensure that the words representing  cosets in the same double coset
have a double coset representative as a common prefix.
For this, one  uses an additional queue $P$, which  like $Q$ initially
contains only the trivial coset  $W_0$, and modifies the processing of
$(x, s)$  so that  a new coset  $z =  x.s$ is also  pushed to the queue $P$, in
addition to $Q$.

The modified algorithm has the same input and output as the original.
The order $\hat{S}$ on $S$ induces an order $\hat{J}$ on
the subset $J$ and an order $\hat{K}$ on its complement $K = S \setminus J$.
The algorithm then proceeds as follows.

1a. Initialize two empty queues  $P$ and $Q$, a vertex list
$V$ and two edge lists $E$ and $T$ as empty lists.
 Then push the trivial coset $W_0$ onto $P$ and $Q$, and add it to $V$.\\
1b. while $P$ is not empty:\\
1c. \qquad pop a coset $y$ off $P$\\
1d. \qquad for $t$ in $\hat{K}$:\\
2a. \qquad \qquad while $Q$ is not empty:\\
3a.  \qquad \qquad \qquad pop a coset $x$ off $Q$\\
4a.  \qquad \qquad \qquad for  $s \in \hat{J}$: process $(x, s)$\\
5a.  \qquad \qquad process $(y, t)$\\
5b. return the graph $\Gamma = (V, E)$ and spanning edges $T$.

Note that this modified algorithm enumerates the cosets of $W_0$ in
$W$ in an order that is potentially different from the original lexicographic order, with
potentially different words for the coset representatives.

In the tables of results below we will indicate which version of the algorithm was used, to uniquely identify the words used for the coset representatives.

\subsection{Inversion of the relations}

The most elementary tool we will use in both cases is the following one.
\begin{Lemma}\label{la:l.s}
Assume that $\alpha \in H_0$ is invertible with inverse $\alpha'$,
and that $\beta \in H$.  Then, for each generator $s$ with inverse 
$s' = q^{-1} s + q^{-1} - 1$, and for any $l, n \in \{1, \dots, \Size{W/W_0}\}$, we have
\begin{align}\label{eq:l.s}
  x_l.s &= \alpha x_n - (q{-}1) \beta &\implies
x_n.s  &= q \alpha'  x_l  + (q{-}1) x_n 
+ q(q{-}1) \alpha' \beta.s', \\\label{eq:l.s'}
x_l.s &= \alpha x_n + (q{-}1) (x_l + \beta) &\implies
x_n.s &= q \alpha' x_l - (q{-}1) \alpha' \beta.s.
\end{align}
\end{Lemma}

Hence, $x_n.s$ can be computed provided that $\beta.s$ is computable.

\begin{proof}
For the first equality, 
we have $x_l.s = \alpha x_n - (q{-}1) \beta$ hence
$\alpha x_n.s'  =   x_l  + (q{-}1) \beta.s'$ and,
expanding $s'$, we get
$x_n.(s - (q{-}1))  =  q \alpha'  (x_l  + (q{-}1) \beta.s' )$.
Therefore, $x_n.s  =  q \alpha'  (x_l  + (q{-}1) \beta.s' ) + (q{-}1) x_n$.
For the second equality we have $x_l.s  = \alpha  x_n  + (q{-}1) x_l  + (q{-}1) \beta$
hence $q x_l.s'  = \alpha  x_n  + (q{-}1) \beta$
and therefore $\alpha  x_n  = q x_l.s' - (q{-}1) \beta$
whence $ x_n.s  = q \alpha' x_l - (q{-}1) \alpha' \beta.s$.
\end{proof}

\subsection{Checking equalities inside the braid group}

The groups $B$ are known to have decidable word problems, and there are actually
efficient decision algorithms. In the case of well-generated reflection groups, Bessis
has shown in \cite{BESSIS}
that the groups $B$ are the groups of fractions of monoids $M$ which share with the
monoid of usual positive braids all the properties used
by Garside to solve the word problem for the usual braid group (such groups $B$ are called Garside groups).
Bessis actually introduced one monoid for each choice of a so-called Coxeter element $c$ in $W$,
and called it the dual braid monoid attached to $c$.
In terms of the generators that we introduce later on (see also the numbering of the diagrams inside Figure~\ref{fig:diagrams}) one can choose $c = s_1 s_2 s_3$ for $G_{24}$ and $G_{27}$, $c = s_1 s_2 s_4 s_3$ for $G_{29}$ and
$c = s_5 s_4 s_2 s_1 s_3$ for $G_{33}$. There
are tools in Michel's development version of the CHEVIE package for GAP3 (which is described in \cite{MIC}) in order to encode
that monoid and therefore to efficiently decide the equalities of two words inside $B$.
In case the groups are badly generated, we use the following properties. 
In the case of $G_{12}$ and $G_{22}$, they are groups of fractions of the monoids $f(4,3)$ and $f(5,3)$, where
$f(h,m)$ denotes the monoid presented by generators $x_1,\dots,x_m$ and
relations
$$
\underbrace{x_1 x_2 \dots x_m x_1 \dots}_{h \ \mathrm{terms} } = 
\underbrace{x_2x_3 \dots x_m x_1 \dots}_{h \ \mathrm{terms} } = \dots
$$
These monoids are also Garside monoids, investigated in M. Picantin's thesis (see \cite{PICTHESE}), and therefore
we can use the
same algorithm to get a normal form. In the case of $G_{31}$, it is possible to embed $B$ inside the Artin group
of type $E_8$, using the formulas of \cite{DMM} \S 3.

Let us now consider some entry $x_l.s$ that we want to compute. 
Let $\tilde{x}_l$ be the image of $x_l$ in $W$. 
There exist $w \in W_0$ and $n \in \{1,\dots,\Size{W/W_0}\}$
such that $\tilde{x}_l.s = w \tilde{x}_n$. Since $W_0$ is a Coxeter group, it is easy to find a shortest length word $m = s_{i_1}\dots s_{i_r}$
representing $w$ in $W$. Then, through the computations of normal forms we can
make $2^r$ tests in order to check
whether the equality $x_l.s = s_{i_1}^{\pm 1}\dots s_{i_r}^{\pm 1}. x_n$ holds inside $B$ for some choice of the signs $\pm 1$.
This is the way we used to find the additional relations used in the sequel.

\section{A sample case by hand : \texorpdfstring{$G_{24}$}{G24}}
\label{sec:G24}

The braid group of type $G_{24}$ admits the presentation
$$
\begin{array}{lcl}
B &=& \langle s_1,s_2,s_3 \ | \ s_1 s_2 s_1 = s_2 s_1 s_2, s_1 s_3 s_1 = s_3 s_1 s_3, s_2s_3s_3s_2 = s_3s_2s_3s_2,\\
& &  s_2 s_3 s_1 s_2 s_3 s_1 s_2 s_3 s_1 = s_3 s_2 s_3 s_1 s_2 s_3 s_1 s_2 s_3 \rangle 
\end{array} 
$$
and the first three relations can be symbolized by the diagram

$$
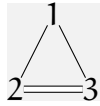
\begin{tikzpicture}[dba/.style={double,double equal sign distance}]
\draw (0,0) node (10) {$2$};
\draw (10)++(1,0) node (11) {$3$};
\draw (10)++(.5,1) node (12) {$1$};
\draw (10) edge[dba] (11);
\draw (10)--(12);
\draw (11)--(12);
\begin{pgfonlayer}{background}
\node [fill=black!5,fit=(current bounding box.north west) (current bounding box.south east)] {};
\end{pgfonlayer}
\end{tikzpicture}
$$

For short, we replace each generator by its numerical label, and
therefore the defining relations for $G_{24}$ become
$121 = 212,
  131 = 313,
  2323 = 3232$
  and
$  231231231 = 323123123$. In the sequel, we will denote $i'$
 the inverse $s_i^{-1}$ of the corresponding generator.

We check equalities inside $B$ using the dual braid monoid. On the example
of $G_{24}$ this can also be done by hand. In order to illustrate this, we prove a few
remarkable identities. For this we first need to describe this monoid.

The dual braid monoid associated to $c = 123$ is generated by 14 atoms $b_1,\dots,b_{14}$
defined by
$$
\begin{array}{lcllcllcllcl}
b_1 &=& 1 & b_2 &=& 2 &
b_3 &=& 3 & b_4 &=& 2'12 \\ b_5 &=& 3'13 & b_6 &=& 232' &
b_7 &=& 3'23 & b_8 &=& 3'b_43 \\ 
b_9 &=& 1232b_8 2'3'2'1' & b_{10} &=& 1b_6 1' &
b_{11} &=& 2'b_{10} 2 & b_{12} &=& 1 b_7 1' \\
b_{13} &=& 2 b_8 2' & b_{14} &=& b_8^{-1} b_2 b_8
\end{array}
$$

It is presented by the relations depicted by all the cycles
of figure \ref{fig:dualrelationsG24} as follows: a cycle $y_1\to y_2 \to \dots \to y_{n+1} = y_1$
represents the relations $y_1y_2 = y_2y_3 = y_3 y_4 = \dots = y_n y_1$. Conjugation by $c$
permutes the atoms as in figure \ref{fig:conjcG24}.

From this we get the relation $231231232= 123123123$. Indeed, $231231232= 23 c^2 2= c^2 (23)^{(c^2)} 2 = c^2 (b_2)^{(c^2)}(b_3)^{(c^2)} b_2 = c^2 b_6 b_{13} b_2$
by the conjugating action of $c$ (see figure \ref{fig:conjcG24}).  Now, by the defining relations of the dual braid monoid (see figure \ref{fig:dualrelationsG24}),
we have $b_6 \underline{b_{13} b_2} = \underline{b_6 b_2} b_8 = b_2 \underline{b_3 b_8} = \underline{b_2 b_4} b_3 = b_1 b_2 b_3 = c$. Notice
that this relation is valid inside $B$, but not inside the monoid with generators $s_1,s_2,s_3$ defined by the same presentation: this proves that this monoid does
not embed in $B$, as opposed to the dual braid monoid. This relation also has a nice group-theoretic interpretation: it says that there exists an involutive automorphism
of $B$ defined by $ 1 \mapsto 1'$,  $2\mapsto 3'$,  $3 \mapsto 2'$.

Another useful relation is $123'2313'23.1 = 2. 123'2313'23$. To prove it, we notice that $3'23 = b_7$ and therefore
$$
\begin{array}{lcl}
123'2313'231
&=& \underline{b_1 b_2} b_7 b_1 b_7 b_1
= b_2 b_4 b_7 \underline{b_1 b_7} b_1
= b_2 b_4 \underline{b_7 b_9} b_{12} b_1 \\
&=& b_2 \underline{b_4 b_1} b_7 b_{12} b_1
= b_2 b_1 b_2 b_7 \underline{b_{12} b_1}
= b_2 b_1b_2b_7b_1b_7 \\
&=& 2 123'2313'23
\end{array}
$$

\begin{figure}
\resizebox{8cm}{!}{
\mbox{$$
\xymatrix{
 & b_8 \ar[r] & b_{12} \ar[r] & b_2 \ar[dd]   \\
b_1 \ar[ur] & & & \\
 & b_{11}\ar[ul] & b_6 \ar[l] & \ar[l] b_5 
 }
\xymatrix{
 & b_9 \ar[r] & b_{13} \ar[r] & b_4 \ar[dd]   \\
b_3 \ar[ur] & & & \\
 & b_{10}\ar[ul] & b_{14} \ar[l] & \ar[l] b_7
 }
$$
}
}
\caption{Action of $x \mapsto c^{-1} x c$ on the set of atoms for type $G_{24}$.}
\label{fig:conjcG24}
\end{figure}
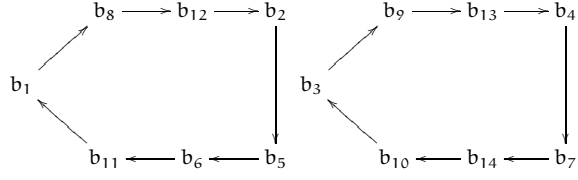
\begin{figure}
\begin{center}
\resizebox{15cm}{!}{
\mbox{
$$
\xymatrix{
b_1 \ar[r] & b_2 \ar[dl] & b_3 \ar[d]  & b_2 \ar[r] & b_{{11}} \ar[dl] & b_7 \ar[d] & b_6 \ar[r] & b_8\ar[dl]  & b_{{14}} \ar[d] & b_{{11}} \ar[r] & b_{{12}} \ar[dl]  & b_9 \ar[r] & b_{{11}} \ar[dl]\\
b_4 \ar[u] & b_1 \ar[ur] & b_5 \ar[l] & b_{{10}} \ar[u] & b_5\ar[ur]  & b_8 \ar[l]  & b_9 \ar[u]  & b_6 \ar[ur]  & b_{{12}}  \ar[l] & b_{{13}} \ar[u] & b_5 \ar[ur] & b_{{14}}\ar[l]\\
b_6 \ar[r] & b_{{13}} \ar[d]  & b_7 \ar[r] & b_9\ar[d]  & b_2\ar[r]  & b_3\ar[d]  & b_2\ar[r]  & b_8\ar[d]  & b_3\ar[r]  & b_8\ar[d]  & b_4\ar[r]  & b_5\ar[d]  \\
b_1\ar[u]  & b_{{10}} \ar[l]  & b_1 \ar[u]& b_{{12}}  \ar[l] & b_6 \ar[u]& b_7 \ar[l]  & b_{{13}}\ar[u] & b_{14} \ar[l]  & b_4\ar[u] & b_{11} \ar[l]  & b_{10}\ar[u] & b_{12} \ar[l]  
}
$$
}}
\end{center}
\caption{Defining relations of the dual braid monoid in type $G_{24}$.}
\label{fig:dualrelationsG24}
\end{figure}
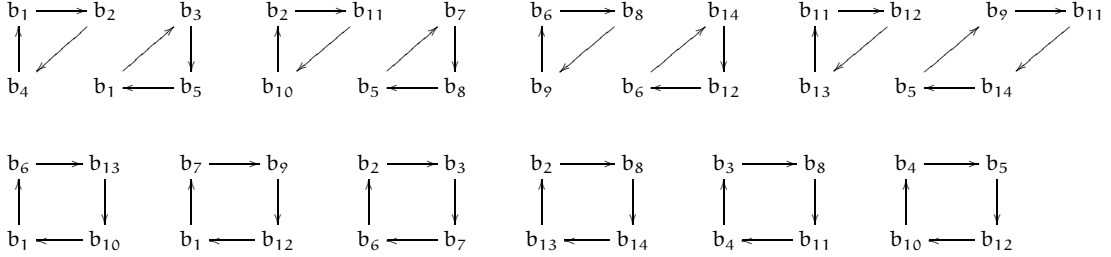

The main computations are gathered in Table~\ref{tableG24}. 
One first gets a list of representatives of the cosets
in the form of words in the generators, as described in the previous section.
Here we choose to group 
the cosets $W/W_0$
corresponding to the same double coset inside $W_0 \backslash W /W_0$,
by using the modified version of the algorithm on the
natural order $\hat{S} = (1,2,3)$.

The entries $x_1.2 = 2 \cdot x_1$ and $x_2.3 = 3 \cdot x_2$
arise from the fact that $W_0$ is generated by $s_2$ and $s_3$.

Entries corresponding to edges in the spanning tree are underlined,
e.g., the edge $x_1 \stackrel{1}{\mbox{-----}} x_2$ is represented by the
entries $\underline{x_{2}}$ for $x_1.1$ and $q \underline{x_{1}} + (q{-}1)x$
for $x_2.1$. (The name $x$ in the entry for $x_n.s$ always denotes $x_n$.)

\begin{table}
\begin{align*}
  \begin{array}{|l|ccc|}
\hline
    x & x.1& x.2 & x.3 \\ \hline
x_{1} = \emptyset & \underline{x_{2}} & 2 \cdot x & 3 \cdot x\\
\hline
x_{2} = 1 & q \underline{x_{1}} + (q{-}1)x & \underline{x_{3}} & \underline{x_{4}}\\
x_{3} = 12 & 2 \cdot x & q \underline{x_{2}} + (q{-}1)x & \underline{x_{5}}\\
x_{4} = 13 & 3 \cdot x & \underline{x_{6}} & q \underline{x_{2}} + (q{-}1)x\\
x_{5} = 123 & \underline{x_{10}} & \underline{x_{7}} & q \underline{x_{3}} + (q{-}1)x\\
x_{6} = 132 & \underline{x_{14}} & q \underline{x_{4}} + (q{-}1)x & \underline{x_{8}}\\
x_{7} = 1232 & \underline{x_{18}} & q \underline{x_{5}} + (q{-}1)x & \underline{x_{9}}\\
x_{8} = 1323 & \underline{x_{22}} & x_9 & q \underline{x_{6}} + (q{-}1)x\\
x_{9} = 12323 & \underline{x_{26}} & q x_8 + (q{-}1)x & q \underline{x_{7}} + (q{-}1)x\\
\hline
x_{10} = 1231 & q \underline{x_{5}} + (q{-}1)x & \underline{x_{11}} & 2 \cdot x \\
x_{11} = 12312 & x_{19} & q \underline{x_{10}} + (q{-}1)x & \underline{x_{12}}\\
x_{12} = 123123 & 232' \cdot x & \underline{x_{13}} & q \underline{x_{11}} + (q{-}1)x\\
x_{13} = 1231232 & \underline{x_{34}} & q \underline{x_{12}} + (q{-}1)x & 2 \cdot x\\
\hline
x_{14} = 1321 & q \underline{x_{6}} + (q{-}1)x & 3 \cdot x & \underline{x_{15}}\\
x_{15} = 13213 & x_{24} & \underline{x_{16}} & q \underline{x_{14}} + (q{-}1)x\\
x_{16} = 132132 & \eqref{16.1} & q \underline{x_{15}} + (q{-}1)x & \underline{x_{17}}\\
x_{17} = 1321323 & \underline{x_{38}} & 3 \cdot x & q \underline{x_{16}} + (q{-}1)x\\
\hline
x_{18} = 12321 & q \underline{x_{7}} + (q{-}1)x & \underline{x_{19}} & \underline{x_{20}}\\
x_{19} = 123212 & q x_{11} + (q{-}1)x & q \underline{x_{18}} + (q{-}1)x & \eqref{19.3} \\
x_{20} = 123213 & x_{28} & \underline{x_{21}} & q \underline{x_{18}} + (q{-}1)x\\
x_{21} = 1232132 & \eqref{21.1} & q \underline{x_{20}} + (q{-}1)x & \eqref{21.3} \\
\hline
x_{22} = 13231 & q \underline{x_{8}} + (q{-}1)x & \underline{x_{23}} & \underline{x_{24}}\\
x_{23} = 132312 & x_{27} & q \underline{x_{22}} + (q{-}1)x & \underline{x_{25}}\\
x_{24} = 132313 & q x_{15} + (q{-}1)x & \eqref{24.2} & q \underline{x_{22}} + (q{-}1)x\\
x_{25} = 1323123 & \eqref{25.1} & \eqref{25.2} & q \underline{x_{23}} + (q{-}1)x\\
\hline
x_{26} = 123231 & q \underline{x_{9}} + (q{-}1)x & \underline{x_{27}} & \underline{x_{28}}\\
x_{27} = 1232312 & q x_{23} + (q{-}1)x & q \underline{x_{26}} + (q{-}1)x & \underline{x_{29}}\\
x_{28} = 1232313 & q x_{20} + (q{-}1)x & \underline{x_{30}} & q \underline{x_{26}} + (q{-}1)x\\
x_{29} = 12323123 & \eqref{29.1} & \underline{x_{31}} & q \underline{x_{27}} + (q{-}1)x\\
x_{30} = 12323132 & \eqref{30.1} & q \underline{x_{28}} + (q{-}1)x & \underline{x_{32}}\\
x_{31} = 123231232 & \eqref{31.1} & q \underline{x_{29}} + (q{-}1)x & \underline{x_{33}}\\
x_{32} = 123231323 & \eqref{32.1} & x_{33} & q \underline{x_{30}} + (q{-}1)x\\
x_{33} = 1232312323 & \underline{x_{42}} & q x_{32} + (q{-}1)x & q \underline{x_{31}} + (q{-}1)x\\
\hline
x_{34} = 12312321 & q \underline{x_{13}} + (q{-}1)x & 232' \cdot x & \underline{x_{35}}\\
x_{35} = 123123213 & 2 \cdot x & \underline{x_{36}} & q \underline{x_{34}} + (q{-}1)x\\
x_{36} = 1231232132 & \eqref{36.1} & q \underline{x_{35}} + (q{-}1)x & \underline{x_{37}}\\
x_{37} = 12312321323 & \eqref{37.1} & 232' \cdot x & q \underline{x_{36}} + (q{-}1)x\\
\hline
x_{38} = 13213231 & q \underline{x_{17}} + (q{-}1)x & \underline{x_{39}} & \eqref{38.3} \\
x_{39} = 132132312 & 3 \cdot x & q \underline{x_{38}} + (q{-}1)x & \underline{x_{40}}\\
x_{40} = 1321323123 & \eqref{40.1} & \underline{x_{41}} & q \underline{x_{39}} + (q{-}1)x\\
x_{41} = 13213231232 & \eqref{41.1} & q \underline{x_{40}} + (q{-}1)x & \eqref{41.3} \\
\hline
x_{42} = 12323123231 & q \underline{x_{33}} + (q{-}1)x & \eqref{42.2} & \eqref{42.3} \\
\hline
\end{array}
\end{align*}
\caption{Multiplication table for $G_{24}$ and sorting}
\label{tableG24}
\end{table}

Some of the remaining entries are straightforward consequences of the braid relations:
\begin{align*}
  x_8.2 &= x_9 & 
  x_9.2 &=  q x_{8} + (q{-}1) x_9 \\
  x_{11}.1 &= x_{19} & 
  x_{19}.1 &= q x_{11} + (q{-}1) x_{19} \\
  x_{15}.1 &= x_{24} &
  x_{24}.1 &= q x_{15} + (q{-}1) x_{24} \\
  x_{20}.1 &= x_{28} &
  x_{28}.1 &= q x_{20} + (q{-}1) x_{28} \\
  x_{23}.1 &= x_{27} &
  x_{27}.1 &= q x_{23} + (q{-}1) x_{27} \\
  x_{32}.2 &= x_{33} &
  x_{33}.2 &= q x_{32} + (q{-}1) x_{33}
\end{align*}
and
\begin{align*}
  x_3.1 &= 2 \cdot x_3 &
  x_4.1 &= 3 \cdot x_4 &
  x_{12}.1 &= 232' \cdot x_{12} \\
  x_{10}.3 &= 2 \cdot x_{10} &
  x_{14}.2 &= 3 \cdot x_{14} &
  x_{34}.2 &= 232' \cdot x_{34} \\
  x_{13}.3 &= 2 \cdot x_{13} &
  x_{17}.2 &= 3 \cdot x_{17} &
  x_{37}.2 &= 232' \cdot x_{37} \\
  x_{35}.1 &= 2 \cdot x_{35} &
  x_{39}.1 &= 3 \cdot x_{39}
\end{align*}
Note that $x_{10}.3$ can also be computed as $x_{10}.3 =   x_{10}.1'3'131$,
there are similar relations for the other equations in this list.

The expression for $x_9.2$ follows obviously by expanding $2'$ in
$x_9.2' = x_8$.  Note that this can also be computed by applying
\eqref{eq:l.s} from Lemma~\ref{la:l.s} with $\alpha = \emptyset$ (the
empty word and identity of $H_0$) and $\beta = 0$.

After that, 19 entries in the table remain to be filled, and 
this is achieved through the following explicit computations.
 
\begin{multline}\label{16.1}
x_{16}.1 = 3'23 \cdot x_{16} 
\\ {}
- (q{-}1)(q3'232' \cdot x_7 + 3'23 \cdot x_9 + 3'23 \cdot x_{15} - q2' \cdot x_{18} - x_{24} - x_{26}) 
\\ {}
+ (q-1)^2(q3'232' \cdot x_5 + 3'23 \cdot x_8 - q2' \cdot x_{10} - x_{22}) 
\end{multline}
In order to get this formula, we start from the relation
\begin{align*}
  13'21'32'.1 = 3'23 \cdot 13'21'32', 
\end{align*}
which holds true
inside $B$. By expansion of the inverses we have $q^2 13'21'3
= x_{15} - (q{-}1)(x_8 + q 2' \cdot x_5)$
and therefore $x_{15} - (q{-}1)(q 2' \cdot x_5 + x_8)).2'1
=  3'23 \cdot (x_{15} - (q{-}1)(q 2' \cdot x_5 + x_8)).2'$.
Expanding $2'$ then yields (\ref{16.1}).

\begin{gather}\label{19.3}
x_{19}.3 = 232' \cdot x_{19} - (q{-}1)(232' \cdot x_{11}  -  x_{12})
\end{gather}
We start from the relation
\begin{align*}
  123121'.3 = 232' \cdot 123121', 
\end{align*}
which holds true in
$B$. By expanding $1'$ it can be rewritten $(x_{19} - (q{-}1)
x_{11}).3 = 232' \cdot (x_{19} - (q{-}1) x_{11})$, from which we get
(\ref{19.3}).

\begin{multline}\label{21.3}
  x_{21}.3 = 232' \cdot x_{21}
-(q{-}1)(
232' \cdot x_{20}
- 2332' \cdot x_{18}
- 2 \cdot x_{12}
+ 23 \cdot x_{11}
) 
\\ {} 
+ (q-1)^2 (23 - 2332') \cdot x_{10}.
\end{multline}
This can be computed as $x_{21}.3 = x_{21}.2'3'23232'$,
or from the relation
\begin{align*}
  1232'13'2'.3 = 232' \cdot 1232'13'2'.
\end{align*}

\begin{gather}\label{24.2}
  x_{24}.2 = 3'23 \cdot x_{24} 
- (q{-}1)(
3'23 \cdot x_{22}
- x_{23}
+ q 3'232' \cdot x_{10}
- q2' \cdot x_{11}
) 
\end{gather}
This can be computed as $x_{24}.2 = x_{24}.1'2'121$, or from the relation
\begin{gather*}
  13'21'31.2 = 3'23 \cdot 13'21'31.
\end{gather*}

\begin{align}
  \label{25.1}
x_{25}.1 &= q^{-2} 23 \cdot x_{36} - (q{-}1)(q^{-2} 23 \cdot x_{35} + q^{-1} 3'23 \cdot x_{34})
\\\label{36.1}
x_{36}.1 &= q^3 3'2' \cdot x_{25} 
+ (q{-}1)(x_{36} + q2' \cdot x_{35} + q^2 23'2' \cdot x_{13})
\end{align}
For the first equation, we use $1323123.1 = 23 \cdot 123123213'2'$ and expand $3'2'$.
The second one is a consequence, multiplying on the right the first one by $1$,
as an application of Lemma~\ref{la:l.s}.

\begin{align}\label{37.1}
  x_{37}.1 &= q^3 3'2' x_{29}
+ (q{-}1) (
x_{37}
+ q 232'2' \cdot x_{35}
+ q^2 x_{13}
) \\\label{29.1}
x_{29}.1 &= q^{-2} 23 \cdot x_{37}
- (q{-}1)(
q^{-1} 23 \cdot x_{34} 
+ q^{-2} 323 \cdot x_{35}
)
\end{align}
The first equation can be computed as $x_{37}.1 = x_{37}.3'1'313$.
The second follows by using the second form of Lemma~\ref{la:l.s}.

\begin{gather}\label{25.2}
  x_{25}.2 = 3'23 \cdot x_{25} - (q{-}1)(3'23 \cdot x_{12} - x_{13})
\end{gather}
By expanding $3'$ we get $13'23123 = x_{25} - (q{-}1) x_{12}$. Then,
multiplying on the right by 2 and using the relation
\begin{gather*}
  13'23123.2 = 3'23 \cdot 13'23123
\end{gather*}
we get (\ref{25.2}).

\begin{multline}
   \label{40.1}
x_{40}.1 = q23 \cdot x_{21} \\
- (q{-}1)(
q 23 \cdot x_{20}
+ q 23 \cdot x_{19}
- q 2'323 \cdot x_{12}
- q^{-2} 223 \cdot x_{37}
-q^{-1} 3'223 \cdot x_{36}
) \\
+ (q{-}1)^2(
q 23 \cdot x_{18} 
- q 323 \cdot x_{10}
- (q^{-1} 3'223 + q^{-2} 2323) \cdot x_{35}
- (3'3'223 + q^{-1} 232) \cdot x_{34}
)
\end{multline}
\begin{multline}\label{21.1}
x_{21}.1 = 3'2' \cdot x_{40} \\
- (q{-}1)(
3' \cdot x_{29} 
- x_{28} 
+ q 3'2'3'2 \cdot x_{25} 
- x_{21} 
+ q 3'2'3 \cdot x_{12}
- q x_{11}
) \\
+ (q{-}1)^2(
q 232' \cdot x_5
- q x_7
- x_{20}
)
\end{multline}
We start from $1'3'2'1323123.1 = 23 \cdot 123213'2' $ and expand the  $1'3'2'$ on the LHS. This provides (\ref{40.1}).
Then (\ref{21.1}) is obtained by multiplying (\ref{40.1}) by $1$ on the right and $q^{-1}3'2'$ on the left, as an application of Lemma~\ref{la:l.s}.

Computing $x_{30}.1 = x_{30}.2'1'212$ yields:
\begin{multline}\label{30.1}
  x_{30}.1 = 3'2' \cdot x_{41}
- (q{-}1)(
3' \cdot x_{31}
- x_{30}
- q x_{20}
-q^2 x_{10}
)
\\
+ (q{-}1)^2(
q x_{11}
+ q^2 23'2' \cdot x_7
- q^2 x_5
)
+ \left(
(q{-}1)^3q^{-1} (232') - (q{-}1)q3'3'2
\right) \cdot x_{25} 
\\
+ \left(
(q{-}1)^3 (3' + q^{-1} 2) - (q{-}1)^2 q^{-1} 2323' - (q{-}1) q 3'2'3 
\right) \cdot x_{13}
\\
+ \left(
(q{-}1)^4 q^{-1} (3-232') + (q{-}1)^2 q 3'3'2 
\right) \cdot x_{12}.
\end{multline}

Computing $x_{31}.1 = x_{31}.2'1212'$ yields:
\begin{multline}\label{31.1}
  x_{31}.1 = 3 \cdot x_{31}
- (q{-}1)(
q 2'3 \cdot x_{25} 
+ 3 \cdot x_{29}
- q^{-1} 232' \cdot x_{36}
- q^{-2} 23 \cdot x_{37}
)
\\
- (q{-}1)^2(
(q^{-2} 323 + q^{-1} 232') \cdot x_{35}
+ q^{-1} 23 \cdot x_{34}
+ 3 \cdot x_{13}
)
\end{multline}

Computing $x_{38}.3 = 1'3'131$ yields:
\begin{multline}\label{38.3}
  x_{38}.3 =  3'23 \cdot x_{38} 
\\
- (q{-}1)(
q^2 3'23 \cdot x_6
- q^2 x_8
+ q 3'23 \cdot x_{18}
- q x_{20}
+ 3'232 \cdot x_{26}
- 2 \cdot x_{28}
)
\\
+ (q{-}1)^2(
q^2 3'23 \cdot x_3
- q^2 x_5
+ 3'223 \cdot x_{22}
-3'23 \cdot x_{24}
)
- (q{-}1)^3(3'223 -3'232) \cdot x_{10}
\end{multline}

Computing $x_{41}.1 = x_{41}.2'1212'$ yields:
\begin{multline}\label{41.1}
    x_{41}.1 = q 23 \cdot x_{30} 
- (q{-}1)(
q^3 23 \cdot x_5
+ q 23 \cdot x_{28}
- 23 \cdot x_{31}
)
\\
+ (q{-}1)^2(
q^{-2} 223 \cdot x_{37}
- 23 \cdot x_{29}
- q 3 \cdot x_{25}
- q 23 \cdot x_{19}
- q^2 3'23 \cdot x_{18}
- 3'223 \cdot x_{12}
)
\\
-(q{-}1)^3 23 \cdot x_{13}
+ ((q{-}1)q^{-1}332 + (q{-}1)^2q^{-1} 23'23) \cdot x_{36}
\\
- ((q{-}1)^2 q^{-1}332 + (q{-}1)^3 q^{-1} 23'23 + (q{-}1)^3q^{-2} 2323) \cdot x_{35}
\\
+ ((q{-}1) q 3 - (q{-}1)^3 q 3' - (q{-}1)^4 3'3'23 - (q{-}1)^3 q^{-1} 232) \cdot x_{34}
\end{multline}

Computing $x_{41}.3 = x_{41}.2'3'2'3232$ yields:
\begin{multline}\label{41.3}
  x_{41}.3 = 3'23 \cdot x_{41}
\\
+ (q{-}1)(
2 \cdot x_{33}
- 2323' \cdot x_{31}
+ q^2 232' \cdot x_{20}
- q^2 23 \cdot x_{18}
+ q^4 x_5
- q^4 3'23 \cdot x_3
)
\\
+ (q{-}1)^2 q(
3'23 \cdot x_{25}
- 3'223 \cdot x_{23}
+ (2332' - 23) \cdot x_{19}
+ (2 - 3'23) \cdot x_{13}
)
\\
+ (q{-}1)^2q^2(
- 3'23 \cdot x_9
+ 3'223 \cdot x_7
)
- (q{-}1)^3 q 3'23 \cdot x_{12}
+ ((q{-}1)^3 3323 - (q{-}1)^4 232) \cdot x_{11}
\end{multline}

Expanding $3'$ inside the braid relation $123'2313'23.1 = 2 \cdot 123'2313'23$
yields:
\begin{multline}\label{32.1}
  x_{32}.1 = 2 \cdot x_{32} 
- (q{-}1)(
q^2 23 \cdot x_5
- q^2 3 \cdot x_{10}
+ 2 \cdot x_{29} 
- q^{-2} 23 \cdot x_{37} 
)\\
+ (q{-}1)^2(
(q\emptyset - q3'2) \cdot x_{12} 
- q^{-1} 23 \cdot x_{34} 
- q^{-2} 323 \cdot x_{35} 
)
\end{multline}

Computing $x_{42}.2 = x_{42}.1'2'121$ yields:
\begin{multline}\label{42.2}
  x_{42}.2 = 2 \cdot x_{42}
- (q{-}1) (
23 \cdot x_{31}
- 223 \cdot x_{29}
)
\\
+ (q{-}1)^2(
q^{-1}(3'23 - 3) \cdot x_{37}
+ (32 - 3'223) \cdot x_{34}
+ q^2 3 \cdot x_{19}
- q^2 23 \cdot x_{18}
)
\\
- ((q{-}1)^3 q^{-2} 323 + (q{-}1)^2 q^{-1} 2232') \cdot x_{36}
\\
+ ((q{-}1)^3 q^{-2} 2323 + (q{-}1)^2q^{-1} 22232') \cdot x_{35}
\end{multline}

Computing $x_{42}.3 = x_{42}.1'3'131$ yields:
\begin{multline}\label{42.3}
  x_{42}.3 = 3 \cdot x_{42}
- (q{-}1)(
q 23 \cdot x_{27}
- q 3'23 \cdot x_{29}
+ q^2 3 \cdot x_{23}
- q^2 x_{25}
)
\\
+ (q{-}1)^2(
q^{-2} (232 - 323) \cdot x_{37}
+ q^{-1} (23 - 2'323) \cdot x_{36}
+ 3'2232' \cdot x_{35}
- 2232' x_{34}
)
\end{multline}

In summary, three different types of operations are used to fill in an
entry.  It is either derived from a suitable relation in the braid
group, or it is derived by replacing the acting generator $s$ by a
word $w$ in the generators (so that $s^{-1}w = 1$ is equivalent to a
defining relation of $W$; this is called a \emph{cyclic expansion} of
$s$ in the next section), or it is obtained by an application of
Lemma~\ref{la:l.s}, that is by \emph{reverting an edge} in the coset
graph.

A systematic search for suitable relations is computationally
expensive and not guaranteed to succeed.  Cyclic expansions and edge
reversals can simply be applied on a trial and error basis.  It turns
out that the operations of cyclic expansion and edge reversal are
sufficient to complete the coset tables for the algebras in
Theorem~\ref{th:maintheo}, provided we add only a few defining relations to the
usual presentations of the braid groups. In the next section we will formulate this
as a strategy.

\section{Algorithm}
\label{sec:algo}

The observations from the example in the previous section can be used
to automate the entire procedure.  This leads to the following
algorithm.  The strategy used is similar to a Todd-Coxeter procedure.
Here, however, first all the cosets are defined all at once (using the information on
cosets in the finite group), and only then cyclic conjugates of the
relations are used to fill missing entries in the table.

By this we mean, that every relation is used to express a generator as
a word in all possible ways.  The words obtained in this way, for a
generator $s \in S$ form the set $R_s$ of \emph{cyclic expansions} of
$s$.

For example, the relation $121=212$, gives cyclic expansions
\begin{align*}
  1&\to 2121'2' &   2&\to 1212'1' \\
  1&\to 2'1'212 &   2&\to 1'2'121 \\
  1&\to 2'1212' &   2&\to 1'2121' 
\end{align*}
i.e., $R_1 = [2121'2',2'1'212,2'1212']$ and
$R_2 = [1212'1',1'2'121,1'2121']$.

The algorithm then proceeds as follows.

0. Compute the lists $R_s$, $s \in S$, of cyclic expansions. \\
1. Compute coset representatives and a spanning tree
as in Section~\ref{sec:coset-graph}, and fill the corresponding entries of the table.\\
2. for each $s \in J$, set the entry $x_1.s = s \cdot x_1$,
where $x_1$ is the trivial coset, represented by the empty word.\\
3. loop over missing entries $x.s$, try to compute $x.s$ as $x.w$ for
$w \in R_s$, or by an application of Lemma~\ref{la:l.s} if possible, until no more
entries can be filled.

Note that step 2 corresponds to filling the subgroup tables in
a Todd-Coxeter procedure.
The order in which the different computations in step 3
are tried is not relevant.

In our implementation of the algorithm, Lemma~\ref{la:l.s} is only
applied, if the coefficient $\alpha$ is obviously invertible, i.e. if
it is a product of a power of $q$ and an element of the natural basis
of $H_0$.  This is sufficient for the purpose of proving
Theorem~\ref{th:maintheo}.  In general, it is indeed a nontrivial task
to identify and invert invertible elements of the Hecke algebra $H_0$.

In the example of $G_{24}$ the algorithm completes after the following
sequence of steps.  Here, an expression like $\revert(x_8.2)$
stands for the result of applying Lemma~\ref{la:l.s} to the known
entry $x_8.2$.

{
\begin{align*}
x_{ 3}.1 &= x_{ 3}.2'1'212 &                           & \dotsc \\  
x_{ 4}.1 &= x_{ 4}.3'1'313 &                      x_{34}.2 &= x_{34}.1'2'121 \\                             
x_{ 8}.2 &= x_{ 8}.3'2'3'2323 &                   x_{35}.1 &= x_{35}.3'1'313 \\                             
x_{ 9}.2 &= \revert(x_{ 8}.2) &                   x_{37}.1 &= x_{37}.3'2'1'3'2'1'232'123123 \\              
x_{10}.3 &= x_{10}.1'3'131 &                      x_{29}.1 &= \revert(x_{37}.1) \\                          
x_{11}.1 &= x_{11}.2'1'212 &                      x_{37}.2 &= x_{37}.3'2'3'2323 \\                          
x_{19}.1 &= \revert(x_{11}.1) &                   x_{39}.1 &= x_{39}.2'1'212 \\                             
x_{12}.1 &= x_{12}.3'2'1'3'2'1'232'123123 &       x_{25}.1 &= x_{25}.3'1313' \\                             
x_{13}.3 &= x_{13}.2'3'2'3232 &                   x_{36}.1 &= \revert(x_{25}.1) \\                          
x_{14}.2 &= x_{14}.1'2'121 &                      x_{25}.2 &= x_{25}.1212'1' \\                             
x_{15}.1 &= x_{15}.3'1'313 &                      x_{31}.1 &= x_{31}.2'1212' \\                             
x_{24}.1 &= \revert(x_{15}.1) &                   x_{32}.1 &= x_{32}.3'2'1'3'2'1'232'123123 \\              
x_{17}.2 &= x_{17}.3'2'3'2323 &                   x_{42}.2 &= x_{42}.1'2'121 \\                             
x_{19}.3 &= x_{19}.1'3131' &                      x_{42}.3 &= x_{42}.1'3'131 \\                             
x_{20}.1 &= x_{20}.3'1'313 &                      x_{24}.2 &= x_{24}.3'23232'3' \\                          
x_{28}.1 &= \revert(x_{20}.1) &                   x_{40}.1 &= x_{40}.3'2'1'23'2'12312313'2' \\              
x_{21}.3 &= x_{21}.2'3'23232' &                   x_{21}.1 &= \revert(x_{40}.1) \\                          
x_{23}.1 &= x_{23}.2'1'212 &                      x_{41}.1 &= x_{41}.2'1212' \\                             
x_{27}.1 &= \revert(x_{23}.1) &                   x_{30}.1 &= \revert(x_{41}.1) \\                          
x_{32}.2 &= x_{32}.3'2'3'2323 &                   x_{41}.3 &= x_{41}.1313'1' \\                             
x_{33}.2 &= \revert(x_{32}.2) &                   x_{16}.1 &= x_{16}.2'1212' \\                      
& \dotsc    &                                   x_{38}.3 &= x_{38}.1'3'131                         
\end{align*}
}

A similar sequence of steps proves the theorem in the remaining number of cases.

\section{Proof of the main theorem}
\label{sec:proof}

The proof of the theorem is then obtained by applying the above algorithm
to each 2-reflection group having a single conjugacy class of reflections, together
with a presentation of the group. We start with the groups of rank 2, where we use the
standard presentations of \cite{BMR}.
In the `ordering'
column we put the ordered set $\hat{S}$ used to build the spanning tree, as in Section~\ref{sec:coset-graph}. We use parenthesis as in $(\mathbf{1},2,3)$ in order to indicate
that we use lexicographic ordering, while we use square brackets as in $[\mathbf{1},2,3]$ to
indicate that we use the modified version of the algorithm that groups cosets into double cosets. In each case,
the digit in bold font indicates (the generator of) the parabolic subgroup which is used -- in general, the digits in bold font will be the generators forming the subset $J$ of Section~\ref{sec:coset-graph}. The other columns
indicate the Coxeter type of the parabolic subgroup $W_0$, the order of the group $W$
and the number of cosets inside $W/W_0$. Finally, the last column contains a checkmark if the
algorithm succeeded, and if not it contains a cross together with the number of entries
that remained empty at the end of the process.

$$
\begin{array}{cccrrcl}
\hline
 W & \mbox{relations} & W_0 & |W| & |W/W_0| & \mbox{ordering} & \mbox{result} \\
 \hline
 \hline
 G_{12} & 1231=2312 &A_1 & 48 & 24 & (\mathbf{1},2,3)& \checkmark \\
 	& 1231 = 3123 & & & &  (1,\mathbf{2},3) & \checkmark \\
 	&  & & & &  (1,2,\mathbf{3})  & \checkmark \\
	&  & & & &  [\mathbf{1},2,3]  & \checkmark \\
		&  & & & &  [1,\mathbf{2},3]  &  \Fail (9) \\
		&  & & & &  [1,2,\mathbf{3}]  & \checkmark \\
\hline	
 G_{22} & 12312=23123 &A_1 &  240 & 120 &(\mathbf{1},2,3)  & \checkmark  \\
 	& 23123 = 31231& & & & (1,\mathbf{2},3)  &\checkmark \\
 	& 12312 = 31231& & & & (1,2,\mathbf{3})  &\checkmark\\
			&  & & & & [\mathbf{1},2,3]  & \Fail (25) \\
			&  & & & & [1,\mathbf{2},3]  &  \Fail (26)\\
			&  & & & & [1,2,\mathbf{3}]  &  \Fail(25) \\
\hline
\end{array}
$$

Of course the choice of a parabolic subgroup matters, in that the completion
of the algorithm proves that $H$ is a free $H_0$-module, for the given choice of $W_0 \subset W$.
The choice of ordering also matters, in that it proves that the specific list of words
in the generators induced by this ordering provides a basis of the free module $H_0$-module $H$.
For instance, let us consider the case where $W$ has type $G_{12}$. In case of $(1,2,3)$, that is the standard lexicographic process
attached to the ordering $(1,2,3)$, the basis of $H$ as a $H_0$-module
that we obtain is

$$
\begin{array}{l}\emptyset,  2 ,  3,  2 1 ,  2 3 ,  3 1 ,  3 2,  2 1 2 , 
   2 1 3 ,  2 3 1 ,  2 3 2 ,  3 1 2 ,  3 1 3 , 
   3 2 1 ,  3 2 3 ,   \\  2 1 2 1 ,2 1 2 3 ,  2 1 3 1 , 
   2 3 2 3 ,  3 1 3 1 ,  3 2 3 2 ,  2 1 2 1 2 , 
   2 1 2 3 2,  2 1 3 1 3 
   \end{array}
   $$
 while for
 $[1,2,3]$ it is
$$
\begin{array}{l} \emptyset, 2 ,  2 1 ,  3 ,  3 1 ,  2 3 ,  2 3 1 ,  2 1 2 , 
   2 1 2 1 ,  2 1 3 ,  2 1 3 1 ,  3 2 ,  3 2 1 , 
   3 1 2 ,  3 1 2 1 ,\\  3 1 3 ,  3 1 3 1 ,  2 3 2 , 
   2 3 2 1 ,  2 1 2 1 2 ,  2 1 2 1 3 ,  2 1 2 1 3 1 , 
   3 1 2 1 2 ,  3 1 2 1 2 1  .
   \end{array}
 $$
Therefore, every checkmark in the table, for a given group, corresponds to a new result,
distinct from the other ones -- but of course only one checkmark is needed in order to
prove Theorem~\ref{th:maintheo} for this group.

We turn to the cases of rank 3 and 4. In the first two cases,
we slightly changed the standard presentation of \cite{BESSISMICHEL}. It is easily checked
that the non-Coxeter relation we use for $G_{24}$ is 
equivalent to the standard one $2 3 1 2 3 1 2 3 1 = 3 2 3 1 2 3 1 2 3$,
while the one we use for $G_{27}$ is equivalent to the
standard one $323123123123 =
 231231231232$.
In the case of $G_{29}$, we do not need any additional relation to the standard presentation.
In the process, we however noticed that the companion relation $423423 = 234234$, which
holds inside the reflection group but \emph{not} in the braid group, admits
a pretty-looking counterpart $42'34'23' = 2'34'23'4$, which holds inside $B$ and might be useful in other contexts.

$$
\begin{array}{cccrrcl}
\hline
 W & \mbox{relations} & W_0 & |W| & |W/W_0| & \mbox{ordering} & \mbox{result} \\
 \hline
 \hline
G_{24} & 121 = 212 &B_2 & 336 & 42 & (1,\mathbf{2},\mathbf{3}) & \checkmark \\
	& 131=313 & & & & [1,\mathbf{2},\mathbf{3}] & \checkmark\\
	& 3232=2323 & \\
	& 12312313' & \\
	& = 232'12312 \\
\hline
G_{27} & 121=212 &B_2  & 2160& 270& (1,\mathbf{2},\mathbf{3}) & \Fail(136) \\
	& 131=313 & & & & [1,\mathbf{2},\mathbf{3}] & \checkmark \\
	& 3232=2323 & \\
 & 232'1231231 & \\
 & = 12312313'23 & \\

 \hline
G_{29} &  121=212 &B_3 & 7680 & 160& (\mathbf{1},\mathbf{2},\mathbf{3},4) & \checkmark \\ 
	& 242 = 424  & & & & [\mathbf{1},\mathbf{2},\mathbf{3},4] & \checkmark \\
	& 343=434 & \\
	& 2323 = 3232 \\
	& 13 = 31, 14=41 \\
	& 432432 = 324324 \\
\hline
\end{array}
$$

The case of the remaining group of rank 4 is somewhat special, in that it involves
two new generators instead of one, and because we needed to introduce a number
of extra relations so that our algorithm manage to fill all the entries of the table. Moreover,
there is no really `natural' ordering of the vertices in this case.  We got the following results.

$$
\begin{array}{cccrrcl}
\hline
 W & \mbox{relations} & W_0 & |W| & |W/W_0| & \mbox{ordering} & \mbox{result} \\
 \hline
 \hline
G_{31} & 141=414, 15=51 &A_3 &46080 & 1920 & (1,\mathbf{2},3,\mathbf{4},\mathbf{5}) & \checkmark \\
	& 242=424, 252=525 & & & & (\mathbf{5},1,\mathbf{2},3,\mathbf{4}) & \checkmark \\
	&34=43, 535 =  353 & & & & (\mathbf{4},\mathbf{5},1,\mathbf{2},3) & \checkmark \\
	& 45=54, 123=231 & & & &(3,\mathbf{4},\mathbf{5},1,\mathbf{2}) & \checkmark \\
	& 231=312,123=312 & & & &(\mathbf{2},3,\mathbf{4},\mathbf{5},1) & \checkmark \\
	& \mathcal{R}_{31} & & & &(\mathbf{2},\mathbf{4},\mathbf{5},1,3) & \checkmark \\
	&  & & & &(\mathbf{2},\mathbf{4},\mathbf{5},3,1) & \checkmark \\
	&  & & & & [1,\mathbf{2},3,\mathbf{4},\mathbf{5}] & \Fail (2633) \\
	\hline
\end{array}
$$
In this table, the additional relations are :
$$
\begin{array}{lcl}
\mathcal{R}_{31} &:& 124124=412412, 235235=523523, 232'523=5232'52, 1242'12=242'124 \\
& &  212'5235=52352'12, 232'4124=41242'32 \\
\end{array}
$$

 Finally, the group $G_{33}$ has a standard presentation in which a parabolic
Coxeter subgroup of type $A_4$ naturally shows up. The Coxeter relations are symbolized by the diagram
$$
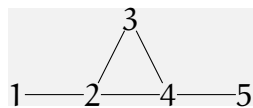
\begin{tikzpicture}[dba/.style={double,double equal sign distance}]
\draw (0,0) node (1) {$1$};
\draw (1)++(1,0) node (2) {$2$};
\draw (2)++(.5,1) node (3) {$3$};
\draw (2)++(1,0) node (4) {$4$};
\draw (2)++(2,0) node (5) {$5$};
\draw (1)--(2);
\draw (2)--(3);
\draw (2)--(4);
\draw (3)--(4);
\draw (4)--(5);
\begin{pgfonlayer}{background}
\node [fill=black!5,fit=(current bounding box.north west) (current bounding box.south east)] {};
\end{pgfonlayer}
\end{tikzpicture}
$$
and there is one additional relation $423423 = 342342$. This
group $G_{33}$
also contains a parabolic subgroup of type $D_4$ which cannot be seen
in this presentation. 
In \cite{BESSISMICHEL}, Bessis and Michel propose an alternative presentation of the
braid group for $G_{33}$, given (up to an harmless swapping of letters) by the Coxeter relations described by the
following diagram,
$$
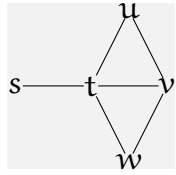
\begin{tikzpicture}[dba/.style={double,double equal sign distance}]
\draw (0,0) node (1) {$s$};
\draw (1)++(1,0) node (2) {$t$};
\draw (2)++(.5,1) node (3) {$u$};
\draw (2)++(.5,-1) node (4) {$w$};
\draw (2)++(1,0) node (5) {$v$};
\draw (1)--(2);
\draw (2)--(3);
\draw (2)--(4);
\draw (2)--(5);
\draw (3)--(5);
\draw (4)--(5);
\begin{pgfonlayer}{background}
\node [fill=black!5,fit=(current bounding box.north west) (current bounding box.south east)] {};
\end{pgfonlayer}
\end{tikzpicture}
$$
together with the relation $wvutwv = vutwvu$.  This presentation is deduced from the previous one
by the relations $s = 1$, $t = 2$, $u = 4$, $v = 3$, $w = 3 4 5 4' 3'$.
From these presentations and the corresponding parabolic subgroups we get the following results, which in particular conclude the proof of Theorem~\ref{th:maintheo}.

\bigskip

$$
\begin{array}{cccrrcl}
\hline
 W & \mbox{relations} & W_0 & |W| & |W/W_0| & \mbox{ordering} & \mbox{result} \\
 \hline
 \hline
G_{33} & 121=212, 323=232&A_4  & 51840 & 432 & (\mathbf{1},\mathbf{2},3,\mathbf{4},\mathbf{5}) & \checkmark\\ 
	& 424=242, 434=343 \\
	& 454=545, 13=31 \\
	&14=41, 15=51 \\
	&  25= 52,35=53 \\
	& 423423 = 342342 \\
	& 342342=234234 \\
\hline	
G_{33} & 121=212,454=545 &D_4  & 51840 & 270 &(\mathbf{1},\mathbf{2},\mathbf{3},4,\mathbf{5})  & \checkmark \\
	& 13=31, 14=41 & & & & [\mathbf{1},\mathbf{2},\mathbf{3},4,\mathbf{5}] & \checkmark \\
	& 15=51, 232=323 &\\
	& 242=424, 252=525 &\\
	& 343=434, 35=53 &\\
	& 543254=432543 &\\
	& \mathcal{R}_{33}^D & \\
\hline
\end{array}
$$

In this table, 
the additional relations are :
$$
\begin{array}{lcl}
\mathcal{R}_{33}^D &:& 324324=432432, 324324=243243,432432=243243,\\ & & 4215421 = 252'421542,425432 = 32543245'\\
\end{array}
$$

Altogether this completes the proof of Theorem~\ref{th:maintheo} for all the 2-reflection
groups, but the case of the largest one, $G_{34}$. This case presents the following obstacles. First of all, because it is so big, we need to use sparse vectors in order
to keep the amount of memory needed to store all entries reasonable. Then, we
know that $G_{34}$ admits two maximal parabolic Coxeter subgroups,
of type $A_5$ and $D_5$, but unfortunately they both have too large index
in $G_{34}$, precisely $54432$ and $20412$, respectively. As a consequence, our
program spends too much time trying to fill in the table : after a few months
of computations it appears to be very far away from the goal. Because of that,
we were not able to prove Proposition~\ref{propGeneH0} in the way we did
before.

We used instead the following `two-layers' strategy. The group $G_{34}$, described by the diagram
below and the additional relation $423423 = 342342$, has a maximal parabolic subgroup $W_0$ of type $G_{33}$, that we use in replacement to the Coxeter subgroups.

$$
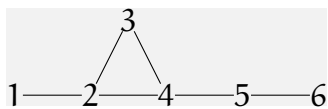
\begin{tikzpicture}[dba/.style={double,double equal sign distance}]
\draw (0,0) node (1) {$1$};
\draw (1)++(1,0) node (2) {$2$};
\draw (2)++(.5,1) node (3) {$3$};
\draw (2)++(1,0) node (4) {$4$};
\draw (2)++(2,0) node (5) {$5$};
\draw (2)++(3,0) node (6) {$6$};
\draw (1)--(2);
\draw (2)--(3);
\draw (2)--(4);
\draw (3)--(4);
\draw (4)--(5);
\draw (5)--(6);
\begin{pgfonlayer}{background}
\node [fill=black!5,fit=(current bounding box.north west) (current bounding box.south east)] {};
\end{pgfonlayer}
\end{tikzpicture}
$$

We let $H_0$
denote the corresponding parabolic Hecke algebra, and try to prove that
$H$ is a $H_0$-module of rank $|W/W_0| = 756$. By the same arguments as before,
this would imply the conjecture for $G_{34}$, and subsequently that $H$ is a free
$H_0$-module of rank $756$.

We define in the usual way a list of coset representatives as words in
the generators $x_1,\dots,x_{756}$, and use the same algorithm as before.
For this we need to use multiplication inside $H_0$. We know how to perform it by
the previous computations, as the table we filled in in the case of $G_{33}$
described $H_0$ as a $(H_{00},  H_0)$-bimodule, where $H_{00}$ is the
parabolic subalgebra of type $A_4$ that we used in this case. We then launched
our algorithm associated with the ordering $(1,2,3,4,5,6)$ and with the Coxeter relations together with $423423 = 342342$
and $ 342342 = 234234$.  This algorithm \emph{almost} completed in a few
hours, in the sense that almost all entries of the table were filled in by then.
At the very end however, the multiplications inside $H_0$
took much more time, because of the size of the entries. Keeping it running,
our program finally completed in 3 weeks, thus proving the remaining $G_{34}$ case.

The interested reader
will find the code we used at the url \url{http://www.lamfa.u-picardie.fr/marin/GGGGcode-en.html}.

\end{document}